\documentclass[a4paper,twoside,12pt,reqno]{amsart}
\usepackage[mode=image]{standalone} 

\makeatletter

\usepackage[margin=4cm]{geometry}
\usepackage[T1]{fontenc}

\usepackage{amsthm, amsmath, amssymb, bbm, bm}
\usepackage{mathrsfs}
\usepackage{yhmath}

\usepackage{accents}

\usepackage{tikz}
\usetikzlibrary{calc}
\usetikzlibrary{intersections}
\usepackage{tikz-3dplot}

\usepackage{enumerate}
\usepackage{graphics}

\usepackage[all]{xy}
\makeatother
\newdir{ >}{{}*!/-10pt/@{>}}
\makeatletter

\usepackage[colorlinks=true, pdfstartview=FitV, linkcolor=blue, %
citecolor=blue, urlcolor=blue]{hyperref}

\usepackage{color}

\usepackage{marginnote}
\usepackage{xspace}
\usepackage{ulem}
\newcommand{\leqnomode}{\tagsleft@true\let\veqno\@@leqno}
\newcommand{\reqnomode}{\tagsleft@false\let\veqno\@@eqno}

\newcommand{\nc}{\newcommand}

\usepackage{units}

\usepackage{accents}

\usepackage{tensor}

\numberwithin{equation}{section}

\theoremstyle{plain}

\newtheorem{theorem}{Theorem}[section]
\newtheorem*{theorem*}{Theorem}

\newtheorem{proposition}[theorem]{Proposition}
\newtheorem{lemma}[theorem]{Lemma}

\theoremstyle{definition}
\newtheorem{definition}[theorem]{Definition}

\newtheorem*{example*}{Example}

\newtheorem{remark}[theorem]{Remark}

\nc{\Lemma}{\begin{lemma}}
\nc{\enlemma}{\end{lemma}}
\nc{\Prop}{\begin{proposition}}
\nc{\enprop}{\end{proposition}}
\nc{\Def}{\begin{definition}}
\nc{\edf}{\end{definition}}

\renewcommand{\emptyset}{\varnothing}

\nc{\scup}{\mathop{\scalebox{.8}{$\displaystyle\bigcup$}}\mspace{1mu}\limits}
\nc{\scap}{\mathop{\scalebox{.8}{$\displaystyle\bigcap$}}\limits}
\nc{\ssqcup}{\mathop{\scalebox{.8}{$\displaystyle\bigsqcup$}}\limits}

\newcommand{\dunion}{\sqcup}

\newcommand{\C}{\mathbb{C}}

\newcommand{\R}{\mathbb{R}}

\newcommand{\Z}{\mathbb{Z}}

\DeclareMathOperator{\id}{id}

\newcommand{\derived}[1]{\mathrm{#1}}

\newcommand{\derd}{\derived{D}}
\newcommand{\dere}{\derived{E}}

\newcommand{\derr}{\derived{R}}
\newcommand{\derl}{\derived{L}}
\nc{\derb}{\derd^{\mathrm{b}}}
\newcommand{\BDC}{\derd^{\mathrm{b}}}

\nc{\soplus}{\scalebox{.65}{\raisebox{.2ex}{$\displaystyle\bigoplus$}}}

\newcommand{\DSum}{\mathop{\bigoplus}}
\newcommand{\dsum}[1][]{\mathbin{\oplus_{#1}}}

\newcommand{\ilim}[1][]{\mathop{\varinjlim}\limits_{#1}}

\renewcommand{\to}[1][]{\xrightarrow{#1}}
\newcommand{\from}[1][]{\xleftarrow{#1}}
\newcommand{\isofrom}[1][]{\xleftarrow[#1]%
{\raisebox{-.4ex}[0ex][-.4ex]{$\mspace{2mu}\sim\mspace{2mu}$}}}
\newcommand{\isoto}[1][]{\xrightarrow[#1]{%
{\raisebox{-.6ex}[0ex][0ex]{$\mspace{1mu}\sim\mspace{2mu}$}}}}

\newcommand{\Endo}[1][]{\mathrm{End}_{\raise1.5ex\hbox to.1em{}#1}}

\newcommand{\Hom}[1][]{\mathrm{Hom}_{\raise1.5ex\hbox to.1em{}#1}}

\newcommand{\RHom}[1][]{\derr\mathrm{Hom}_{\raise1.5ex\hbox to.1em{}#1}}

\newcommand{\Ext}[2][]{\mathrm{Ext}_{\raise1.5ex\hbox to.1em{}#1}^{#2}}

\newcommand{\Mod}{\mathrm{Mod}}

\newcommand{\Tens}[1][]{\mathbin{\otimes_{\raise1.5ex\hbox to-.1em{}#1}}}

\newcommand{\LTens}[1][]{\mathbin{\otimes_{\raise1.5ex\hbox to-.1em{}#1}^{\derl}}}

\newcommand{\Tor}[2][]{\mathrm{Tor}^{\raise1.5ex\hbox to.1em{}#1}_{#2}}

\newcommand{\sheaffont}[1]{\mathcal{#1}}

\def\she{\sheaffont{E}}

\def\shi{\sheaffont{I}}

\def\shl{\sheaffont{L}}
\def\shm{\sheaffont{M}}

\def\shp{\sheaffont{P}}

\newcommand{\rsect}{\derr\varGamma}

\newcommand{\shendo}[1][]{{\sheaffont{E}nd}_{\raise1.5ex\hbox to.1em{}#1}}

\renewcommand{\hom}[1][]{{\sheaffont{H}om}_{\raise1.5ex\hbox to.1em{}#1}}
\newcommand{\aut}[1][]{{\sheaffont{A}ut}_{\raise1.5ex\hbox to.1em{}#1}}
\newcommand{\inn}[1][]{{\sheaffont{I}nn}_{\raise1.5ex\hbox to.1em{}#1}}

\newcommand{\rhom}[1][]{{\derr\sheaffont{H}om}_{\raise1.5ex\hbox to.1em{}#1}}

\newcommand{\ext}[2][]{{\sheaffont{E}xt}_{\raise1.5ex\hbox to.1em{}#1}^{#2}}

\newcommand{\thom}[1][]{{\sheaffont{T}hom}_{\raise1.5ex\hbox to.1em{}#1}}

\newcommand{\tens}[1][]{\mathbin{\otimes_{\raise1.5ex\hbox to-.1em{}#1}}}

\newcommand{\ltens}[1][]{\mathbin{\otimes_{\raise1.5ex\hbox to-.1em{}#1}^{\derl}}}

\newcommand{\tor}[2][]{{\sheaffont{T}or}^{\raise1.5ex\hbox to.1em{}#1}_{#2}}

\newcommand{\etens}[1][]{\mathbin{\boxtimes_{\raise1.5ex\hbox to-.1em{}#1}}}

\newcommand{\oim}[1]{#1_*}

\newcommand{\roim}[1]{\derr#1_*}

\newcommand{\reim}[1]{\derr#1_{\mspace{.5mu}!}\mspace{2mu}}

\newcommand{\reeim}[1]{\derr#1_{\mspace{1mu}!!}\mspace{1mu}}

\newcommand{\opb}[1]{#1^{-1}}

\newcommand{\epb}[1]{#1^{\mspace{1.5mu}!}\mspace{2mu}}

\newcommand{\tenstop}[1][]{\mathbin{\hat{\otimes}_{\raise1.5ex\hbox to-.1em{}#1}}}

\newcommand{\homtop}[1][]{\sheaffont{L}_{\raise1.5ex\hbox to.1em{}#1}}

\newcommand{\Homtop}[1][]{\mathrm{L}_{\raise1.5ex\hbox to.1em{}#1}}

\DeclareMathOperator{\ord}{ord}

\newcommand{\D}{\sheaffont{D}}

\renewcommand{\O}{\sheaffont{O}}

\newcommand{\detens}[1][]%
{\mathbin{\boxtimes_{\raise1.5ex\hbox to-.1em{}#1}^{\mspace{2mu}\mathsf{D}}}}

\newcommand{\dtens}[1][]{\mathbin{\otimes_{\raise1.5ex\hbox to-.1em{}#1}^{\mathsf{D}}}}

\newcommand{\hol}{\mathrm{hol}}

\newcommand{\reghol}{{\mathrm{rh}}}

\renewcommand{\leq}{\leqslant}
\renewcommand{\geq}{\geqslant}

\newcommand{\field}{\mathbf{k}}
\newcommand{\ind}{\mathrm{I}\mspace{2mu}}
\newcommand{\ifield}{\ind\field}

\newcommand{\Rc}{{\R\text-\mathrm{c}}}
\newcommand{\Cc}{{\C\text-\mathrm{c}}}

\newcommand{\cl}{\colon}

\newcommand{\ctens}{\mathbin{\mathop\otimes\limits^+}}
\newcommand{\cetens}[1][]{\mathbin{\mathop\boxtimes\limits^+_{\raise1.5ex\hbox to-.1em{}#1}}}
\newcommand{\cihom}{{\derr\shi hom}^+}

\newcommand{\dr}{\mathcal{DR}}

\renewcommand{\Re}{\operatorname{Re}}
\renewcommand{\Im}{\operatorname{Im}}

\newcommand{\ihom}[1][]{{\shi hom}_{\raise1.5ex\hbox to.1em{}#1}}
\newcommand{\rihom}[1][]{{\derr\mspace{2mu}\shi hom}_{\raise1.5ex\hbox to.1em{}#1}}
\newcommand{\ii}[1][]{{\sheaffont{I}h}_{\raise1.5ex\hbox to.1em{}#1}}

\newcommand{\indlim}[1][]{\mathop{\text{\rm``$\varinjlim$''}}\limits_{#1}}

\newcommand{\dcomp}[1][]{\mathbin{\circ_{\raise1.5ex\hbox to-.1em{}#1}^{\mathsf{D}}}}

\newcommand{\enh}{\derived{E}}

\newcommand{\fhom}{\rhom^\enh}
\newcommand{\FHom}{\RHom^\enh}

\newcommand{\BEC}[2][\ifield]{\dere^{\mathrm{b}}(#1_{#2})}
\newcommand{\BECRc}[2][\ifield]{\dere^{\mathrm{b}}_\Rc(#1_{#2})}

\newcommand{\BECcon}[2][\ifield]{\dere^{\mathrm{b}}_{\bGmp}(#1_{#2})}
\newcommand{\BECp}[2][\ifield]{\dere^{\mathrm{b}}_+(#1_{#2})}

\newcommand{\BECm}[2][\ifield]{\dere^{\mathrm{b}}_-(#1_{#2})}
\newcommand{\BECpm}[2][\ifield]{\dere^{\mathrm{b}}_\pm(#1_{#2})}

\newcommand{\Edual}{\dual^\enh}
 \newcommand{\Eoim}[1]{\enh#1_*}

\newcommand{\Eeeim}[1]{\enh#1_{!!}}

\newcommand{\Eopb}[1]{\enh#1^{-1}}
\newcommand{\Eepb}[1]{\enh\mspace{1mu}#1^{\mspace{1.5mu}!}}

\newcommand{\LE}{\operatorname{L^\enh}}

\newcommand{\semicolon}{\nobreak \mskip2mu\mathpunct{}\nonscript\mkern-\thinmuskip{;}\mskip6mu plus1mu\relax}

\newcommand{\dual}{\mathrm{D}}

\newcommand{\defeq}{\mathbin{:=}}
\newcommand{\eqdef}{\mathbin{=:}}
\newcommand{\bl}{\bigl(}
\newcommand{\br}{\bigr)}
\newcommand{\To}[1][]{\xrightarrow[]{\mspace{10mu}{#1}\mspace{10mu}}}

\newenvironment{myarray}[1]{\relax\setlength{\arraycolsep}{1pt}

\begin{array}{#1}}{\end{array}\relax}

\newcommand{\ba}{\begin{myarray}}
\newcommand{\ea}{\end{myarray}}

\newcommand{\be}{\begin{enumerate}}
\newcommand{\ee}{\end{enumerate}}
\newcommand{\bnum}{\be[{\rm(i)}]}
\newcommand{\bna}{\be[{\rm(a)}]}
\nc{\bwr}{\mbox{\large{$\wr$}}}
\nc{\vphi}{\varphi}
\nc{\seteq}{\mathbin{:=}}
\nc{\noi}{\noindent}
\nc{\ro}{{\rm(}}
\nc{\rf}{{\rm)}\xspace}
\nc{\ms}{\mspace}
\nc{\sbcup}{\mathop{\scalebox{0.75}{$\displaystyle\bigcup$}}}
\nc{\ol}{\overline}
\nc{\scbul}{{\,\raise1pt\hbox{$\scriptscriptstyle\bullet$}\,}}
\nc{\set}[2]{\left\{#1\;\semicolon\; #2 \right\}}
\nc{\extp}{\mathop{\raisebox{.3ex}{\scalebox{0.8}{$\displaystyle\bigwedge$}}}\limits}

\newenvironment{myequation}
{\relax\setlength{\arraycolsep}{1pt}\begin{eqnarray}}
{\end{eqnarray}}
\newenvironment{myequationn}
{\relax\setlength{\arraycolsep}{1pt}\begin{eqnarray*}}
{\end{eqnarray*}}

\newenvironment{myalign}
{\relax\begin{align}}
{\end{align}}
\newenvironment{myalignn}
{\relax\begin{align*}}
{\relax\end{align*}}

\nc{\eq}{\begin{myequation}}
\nc{\eneq}{\end{myequation}}
\nc{\eqn}{\begin{myequationn}}
\nc{\eneqn}{\end{myequationn}}

\nc{\eqa}{\begin{myalign}}
\nc{\eneqa}{\end{myalign}}
\nc{\eqan}{\begin{myalignn}}
\nc{\eneqan}{\end{myalignn}}

\nc{\on}{\operatorname}
\nc{\Ind}{\on{Ind}}
\nc{\Proof}{\begin{proof}}
\nc{\QED}{\end{proof}}
\nc{\cor}{\field}
\nc{\tone}{\To[+1]}
\renewcommand{\ge}{\geq}
\renewcommand{\le}{\leq}

\newcommand{\BECCc}[2][\ifield]{\dere^{\mathrm{b}}_\Cc(#1_{#2})}

\renewcommand{\tor}{\mathrm{tor}}

\newcommand{\bclose}[1]{{\accentset{\vee}{#1}}}
\newcommand{\unbordered}[1]{{\accentset{\circ}{#1}}}

\newcommand{\bR}{{\R_\infty}}

\nc{\unb}{\unbordered}
\nc{\eps}{\varepsilon}
\nc{\epsi}{\epsilon}
\nc{\inb}{\inbordered}
\nc{\colim}{\varinjlim\limits}
\nc{\ssubset}{\subset\ms{-3mu}\subset}
\nc{\al}{\alpha}
\nc{\qtq}[1][and]{\quad\text{#1}\quad}
\nc{\qt}[1]{\quad\text{#1}}
\nc{\olG}[1][f]{{\overset{\ms{4mu}\rule[-.05ex]{1.6ex}{.115ex}}{\Gamma}}_{%
\ms{-3mu}#1}}

\nc{\cf}{\bclose{f}}
\nc{\cp}{\bclose{p}}

\newcommand{\inbordered}[1]{{#1_\infty}}

\newcommand{\bZ}{\inbordered{Z}}

\newcommand{\quot}{\derived Q}

\newcommand{\Efield}{\field^\enh}

\newcommand{\V}{\mathbb{V}}
\newcommand{\W}{\V^*}
\newcommand{\PP}{\mathbb{P}}

\newcommand{\bb}{\PP^*}

\newcommand{\Ex}{\mathbb{E}}
\newcommand{\ex}{\mathsf{E}}
\newcommand{\exx}{\ex}

\newcommand{\Fou}{\mathsf{F}}

\newcommand{\Lap}{\mathsf{L}}
\newcommand{\lap}{{}^\Lap}
\newcommand{\pLap}{\mathbb{L}}
\newcommand{\plap}{{}^\pLap}

\newcommand{\Perv}{\operatorname{Perv}}
\newcommand{\Eperv}{\enh\text{-}\!\Perv}

\nc{\tM}{\widetilde{M}}
\nc{\tX}{\widetilde{X}}
\nc{\ti}{{\tilde\imath}}
\nc{\tj}{{\tilde\jmath}}

\newcommand{\dt}[1]{\widetilde{#1}^{\mathsf d}}
\newcommand{\dtd}[1]{\widetilde{#1}^{\mathsf d,\times}}
\nc{\dtM}{\dt M}
\nc{\dtdM}{\dtd M}
\nc{\dti}{\tilde\imath^{\mathsf d}}

\newcommand{\eprec}{\preccurlyeq}
\newcommand{\aleq}[2][]{\mathrel{\preccurlyeq^{#1}_{#2}}}

\newcommand{\range}[2]{#1 \mathbin{\rhd} #2}

\newcommand{\dotowns}{\mathbin{\bdot\owns}}
\newcommand{\dotsupset}{\mathbin{\bdot\supset}}

\newcommand{\oshp}{\smash{\overline\shp}}

\newcommand{\rb}{\mathsf{rb}}
\newcommand{\pb}{\mathsf{pb}}
\newcommand{\nd}{\mathsf{nd}}
\newcommand{\sm}{\mathsf{sm}}
\newcommand{\sph}{\mathsf{sph}}
\nc{\st}[1]{{\{{#1}\}}}
\nc{\bP}{\mathbb{P}}
\nc{\into}{\hookrightarrow}
\nc{\fR}{{\R_\infty}}
\nc{\cS}{\bclose{S}}
\nc{\RB}[2][N]{#2_{#1}^{\rb}}
\nc{\prb}{p_{\rb}}
\nc{\PB}[2][N]{#2_{#1}^{\pb}}
\nc{\ppb}{p_{\pb}}
\nc{\ND}[2][N]{#2_{#1}^{\nd}}
\nc{\pnd}{p_{\nd}}
\nc{\snd}{s_{\nd}}
\nc{\dND}[2][N]{\bdot{#2}_{#1}^\nd}
\nc{\psm}{p_{\sm}}

\newcommand{\sh}{\mathsf{sh}}

\newcommand{\near}[1][a]{\Psi_{#1}}
\newcommand{\van}[1][a]{\Phi_{#1}}

\newcommand{\fnear}[2][a]{\near[#1]^{\eprec #2}}
\newcommand{\gnear}[2][a]{\near[#1]^{#2}}

\newcommand{\Enu}{\enh\nu}
\newcommand{\Emu}{\enh\mu}
\newcommand{\Esm}{\enh\sigma}

\newcommand*\bigcdot{\mathpalette\bigcdot@{.5}}
\newcommand*\bigcdot@[2]{\mathbin{\vcenter{\hbox{\scalebox{#2}{$\m@th#1\bullet$}}}}}
\newcommand{\bdot}[1]{{\accentset{\bigcdot}{#1}}\vphantom{#1}}

\newcommand{\Gmp}{\R^\times_{>0}}
\newcommand{\bGmp}{\inb{(\Gmp)}}

\newcommand{\SV}{\mathbb{S}\V}

\newcommand{\bbM}{\mathsf{M}}
\newcommand{\ccM}{C}
\newcommand{\obM}{\unbordered{\bbM}}
\newcommand{\cbM}{\bclose{\bbM}}

\newcommand{\bbN}{\mathsf{N}}
\newcommand{\obN}{\unbordered{\bbN}}
\newcommand{\cbN}{\bclose{\bbN}}

\newcommand{\oo}{\unbordered}
\newcommand{\of}{\oo f}
\nc{\oloG}[1][\of]{{\overset{\ms{4mu}\rule[-.05ex]{1.6ex}{.115ex}}{\Gamma}}_{%
\ms{-3mu}#1}}

\nc{\ake}[1][2ex]{\rule[-.5ex]{0ex}{#1}}
\nc{\akew}[1][2ex]{\rule[-1ex]{#1}{0ex}}
\nc{\aked}[2][2ex]{\rule[{#1}]{0ex}{#2}}

\nc{\e}{\mathrm{e}}
\nc{\vep}{\varepsilon}

\makeatother

\begin{document}

\title[Enhanced nearby and vanishing cycles]{Enhanced nearby and vanishing cycles in dimension one and Fourier transform}

\author[A.~D'Agnolo]{Andrea D'Agnolo}
\address[Andrea D'Agnolo]{Dipartimento di Matematica\\
Universit{\`a} di Padova\\
via Trieste 63, 35121 Padova, Italy}
\thanks{The research of A.D'A.\
was partially supported by GNAMPA/INdAM.  He acknowledges the kind hospitality at RIMS of
Kyoto University, and at the Perimeter Institute in Waterloo, during the preparation of this paper.}
\email{dagnolo@math.unipd.it}

\author[M.~Kashiwara]{Masaki Kashiwara}
\thanks{The research of M.K.\
was supported by Grant-in-Aid for Scientific Research (B)
15H03608, Japan Society for the Promotion of Science}
\address[Masaki Kashiwara]{
Kyoto University Institute for Advanced study,
Research Institute for Mathematical Sciences, Kyoto University,
Kyoto 606-8502, Japan \& Korea Institute for Advanced Study, Seoul 02455, Korea}
\email{masaki@kurims.kyoto-u.ac.jp}


\maketitle

\begin{abstract}
Enhanced ind-sheaves provide a suitable framework for the irregular Riemann-Hilbert correspondence.
In this paper, we give some precisions on nearby and vanishing cycles for enhanced perverse objects in dimension one. As an application, we give a topological proof of the following fact.
Let $\shm$ be a holonomic algebraic $\D$-module on the affine line, 
and denote by $\lap\shm$ its Fourier-Laplace transform.
For a point $a$ on the affine line, denote by $\ell_a$ the corresponding linear function on the dual affine line.
Then, the vanishing cycles of $\shm$ at $a$ are isomorphic to 
the graded component of degree $\ell_a$ of the Stokes filtration of $\lap\shm$ at infinity.
\end{abstract}

\section{Introduction}

\subsection{}\label{sse:introphipsi}
Let $X$ be a smooth complex curve and $F$ a perverse sheaf  on $X$. Recall that, near a singularity $a\in X$, $F$ admits a quiver description in terms of its nearby and vanishing cycles $\near(F)$ and $\van(F)$.  
Let $S_a X$ and $S^*_a X$ be the circles of tangent and cotangent directions at $a$, respectively.
Using the canonical and variation maps $\xymatrix@1{\near(F) \ar@<.3ex>[r]^c &  \ar@<.3ex>[l]^v \van(F),}$ one may upgrade the vector spaces $\near(F)$ and $\van(F)$ to local systems on $S_a X$ and $S^*_a X$, with monodromies $1-vc$ and $1-cv$, respectively. Then, one has
\begin{equation}\label{eq:intronupsi}
\near(F) \simeq \nu^\sph_{\st a}(F), \quad
\van(F) \simeq \mu^\sph_{\st a}(F)[1],
\end{equation}
where $\nu^\sph_{\st a}$ and $\mu^\sph_{\st a}$ denote the traces on $S_a X$ and $S^*_a X$ of Sato's specialization $\nu_{\st a}$ and microlocalization $\mu_{\st a}$ functors, respectively.

\subsection{}\label{sse:rb}
Let $\RB[a]X$ be the real oriented blow-up of $X$ with center $a$, and consider the natural embeddings
\begin{equation}\label{eq:introRB}
\xymatrix@R=2ex{
S_a X \ar@{^(->}[r]^-i & \RB[a] X
& X\setminus\{a\} \ar@{^(->}[r]^-{j_a} \ar@{_(->}[l]_-j & X.
}
\end{equation}
Recall that one has
\[
\nu_{\st a}^\sph(F) \simeq \opb i \roim j \opb j_a F,
\]
where $\opb i$, $\roim j$ and $\opb{j_a}$ denote the external operations for sheaves.

\subsection{}
Let $\shm$ be a (not necessarily regular) holonomic $\D_X$-module, 
 let $a\in X$ be one of its singularities, 
and let $F\defeq\dr(\shm)$ be its de Rham complex, which is a perverse sheaf. 
If $\shm$ is regular, the classical Riemann-Hilbert correspondence implies that $\shm$ can be reconstructed  near $a$  from the quiver description of $F$. If $\shm$ is irregular, a result of Deligne and Malgrange (see \cite{DMR}) implies that $\shm$ can be reconstructed  near $a$  by further considering the so-called Stokes filtration\footnote{The Stokes filtration depends on $\shm$, and not only on $F$.} $\fnear \bullet(F,\shm)$ of $\near(F)$, indexed by Puiseux germs, defined as follows.
Let $(a,\theta,f)$ be a Puiseux germ, that is, a holomorphic function $f$ on a small sector around $\theta\in S_aX$, which admits a Puiseux series expansion at $a$.
For $(a,\theta,g)$ another germ, the order relation $g\aleq\theta f$ means that $\Re(g-f)$ is bounded from above on a small sector around $\theta$.
Then, an element $u$ of the stalk $\bl\fnear f(F,\shm)\br_\theta$ is a section of the de Rham complex of $\shm$ in a sectorial neighborhood of $\theta$ such that $\e^{-f}u$ has tempered growth at $a$. It turns out that the graded component $\gnear f(F,\shm)$ is a locally constant sheaf on $S_a X$.

\subsection{}\label{sse:RH}
In \cite{DK16} we established an extension of the classical Riemann-Hilbert correspondence to the irregular case, in the framework of enhanced ind-sheaves, which has the advantage of working in any dimension.
More precisely, there is a quasi-commutative diagram
\begin{equation}\label{eq:introRH}
\xymatrix@R=3ex@C=3em{
\Mod_\reghol(\D_X) \ar@{ >->}[d]^\iota \ar[r]^\dr_\sim & \Perv(\C_X) \ar@{ >->}[d]^{e\,\iota} \\
\Mod_\hol(\D_X) \ar[r]^{\dr^\enh}_\sim & \Eperv(\ind\C_X).
}
\end{equation}

Here, $\iota$ embeds regular holonomic $\D$-modules into holonomic $\D$-modules which are not necessarily regular, $e\,\iota$ embeds perverse sheaves into enhanced ones\footnote{This refers to the present case of dimension one. In higher dimension there is still no explicit description of the category of perverse enhanced sheaves. See however \cite{Moc16}, where such a category is described via a curve test.} (see Definition~\ref{def:ECc}), and $\dr^\enh$ is an enhancement\footnote{As $\dr^\enh$ is only briefly mentioned in this paper, we do not recall its definition, referring instead to \cite{DK16}.} of the de Rham functor $\dr$.

\subsection{}
In \cite[\S6.2]{DK18}  we described the Stokes filtration $\fnear \bullet(F,\shm)$  in terms of the enhanced de Rham complex $\dr^\enh(\shm)$. Here, using enhanced specialization and microlocalization from \cite{DK21a}, and making a more explicit use of the sheafification functor discussed in \cite{DK21b}, we propose a description of the Stokes filtration which sheds some light on the geometry underlying these constructions. We also discuss a tempered version of the vanishing cycles $\van(F)$ as follows.

\subsection{}
Let $\field$ be a field, and consider the natural embedding $e\,\iota\colon\BDC(\field_X)\to[\iota]\BDC(\ifield_X)\to[e]\BEC X$, of sheaves into ind-sheaves into enhanced ind-sheaves. Recall that $e\,\iota$ has a left quasi-inverse $\sh$ called sheafification functor.
We say that an enhanced ind-sheaf $K$ is of sheaf type if it lies in the essential image of $e\,\iota$. 

\subsection{}\label{sse:intropsiphi}
Let $K\in\Eperv(\ifield_X)$, and let $a\in X$ be a singularity of $K$. 
The nearby and vanishing cycles of $K$ are defined as follows.
Consider the bordered analogue of \eqref{eq:introRB}
\[
\xymatrix@R=2ex{
S_a X \ar@{^(->}[r]^-i & \RB[a] X
& \inb{(X\setminus\{a\})} \ar@{^(->}[r]^-{j_a} \ar@{_(->}[l]_-j & X,
}
\]
where $\inb{(X\setminus\{a\})}$ denotes the bordered space $(X\setminus\{a\},X)$.
We set
\begin{align*}
\near(K) 
&\defeq \nu_{\st a}^\sph\bl\sh(K)\br \\
&\simeq \opb i \roim j \opb j_a\sh (K) \\
&\simeq \opb i \roim j \sh \bl\Eopb j_a K\br, \\
\fnear 0 (K) 
&\defeq \opb i \sh \bl\Eoim j \Eopb j_a K\br, \\
 \gnear 0 (K)  &  \defeq \sh\bl \Eopb i \Eoim j \Eopb j_a K \br \\
&  \simeq \sh \bl \Enu_{\st a}^\sph(K)\br. 
\end{align*}
Here, $\Eopb i$, $\Eoim j$ and $\Eopb{j_a}$ denote the external operations for enhanced ind-sheaves, 
and  $\Enu_{\st a}$  is the natural enhancement of Sato's specialization.
Further, for $(a,\theta,f)$ a Puiseux germ, locally at  $\theta$  set 
\begin{align*}
\fnear f (K) &\defeq \fnear 0 \bl K(f)\br, \\
\gnear f (K) &\defeq \gnear 0 \bl K(f)\br, 
\end{align*}
where $K(f)$ is the twist of $K$ by an enhanced ind-sheaf which encodes the exponential growth $\e^f$ (see Definition~\ref{def:psiphi}). Finally, set
\begin{align*}
\van(K) &\defeq \mu_{\st a}^\sph\bl\sh(K)\br[1], \\
\van^0(K) &\defeq \sh\bl\Emu_{\st a}^\sph(K)\br[1],
\end{align*}
where $\Emu_{\st a}$ is the natural enhancement of Sato's microlocalization.

As it turns out, $\Enu_{\st a}(K)$ and $\Emu_{\st a}(K)$ are of sheaf type.

\subsection{}
If $\field=\C$, and $K=\dr^\enh(\shm)$ is the enhanced de Rham complex of a holonomic $\D_X$-module $\shm$ then, setting $F\seteq\dr(\shm)\simeq\sh(K)$, one has by definition
\[
\near(K)\simeq\near(F), 
\quad \van(K)\simeq\van(F).
\]
Moreover, one has
\[
\fnear\bullet (K) \simeq \fnear\bullet (F,\shm),
\quad \gnear\bullet (K) \simeq \gnear\bullet (F,\shm).
\]
Also note that $\van^0(K) \simeq\van(F)$ only if $\shm$ is regular.

\subsection{}\label{sse:intropsiphiM}
Recall that an exponential factor at $a$ of a holonomic $\D_X$-module $\shm$ is a Puiseux germ $(a,\theta,f)$ where the Stokes filtration $\fnear\bullet(F,\shm)$ jumps. Assume for simplicity that the exponential factors of $\shm$ are unramified, so that $f$ is a germ of meromorphic function with  a pole  at $a$.
Let $N_\theta^{>0}$ be a set of representatives of the exponential factors of $\shm$, modulo bounded functions. We can assume that if $f\in N_\theta^{>0}$ is bounded, then $f=0$. 

For $f$ unbounded, let $\she^f$ be the germ of $\D_X$-module at $a$ associated with the meromorphic connection $d+df$. Set $\she^0=\O_X$.
The Hukuara-Levelt-Turrittin decomposition theorem asserts that
\[
\shm^{\operatorname{formal}}_a \simeq \DSum_{f\in N_\theta^{>0}} \bl \shl_f\dtens\she^{-f} \br^{\operatorname{formal}}_a.
\]
Here, $\shm^{\operatorname{formal}}_a$ is the formal $\D$-module at $a$ associated with $\shm$, $\dtens$ is the inner product for $\D_X$-modules, and $\shl_f$ is a regular holonomic $\D_X$-module. 

Set $K=\dr^\enh(\shm)$ and $L_f=\dr(\shl_f)$. 
Then, we have
\[
\Enu_{\st a}(K) \simeq e\bl\nu_{\st a}(L_0)\br, \quad
\Enu^\sph_{\st a}\bl K(f)\br \simeq e\bl\nu^\sph_{\st a}(L_f)\br.
\]

\subsection{}
We give an application of the above constructions to the study of the Fourier-Laplace transform in dimension one.

Let $\V$ be a one-dimensional complex vector space, and $\W$ the dual vector space.
Set $\V_\infty \defeq (\V,\PP)$, where $\PP=\V\cup\{\infty\}$ is the 
projective compactification, and similarly define $\W_\infty$ and $\bb$.

Let $K\in\Eperv(\ifield_{\V_\infty})$, and set $\plap K\defeq\lap K[1]$.
(Here, the shift ensures compatibility with the Riemann-Hilbert correspondence.)
Assume that $\plap K$ is an enhanced perverse ind-sheaf on $\W_\infty$.
For $a\in\PP$, let $(a,\theta,f)$ be a Puiseux germ on $\PP$ such that $f$ is unbounded and \emph{not linear} (modulo bounded functions). Then its Legendre transform $(b,\eta,g)$ is a Puiseux germ on $\bb$ of the same kind.
The stationary phase formula states that there is an isomorphism
\begin{equation}
\label{eq:introPhiPhi}
\bl\gnear[b] g(\plap K)\br_\eta \simeq \bl\gnear f(K)\br_\theta.
\end{equation}
This is a classical result for holonomic $\D$-modules, and we gave a proof for enhanced ind-sheaves in \cite{DK18}.

\subsection{}
Here we consider the case of linear Puiseux germs, excluded from \eqref{eq:introPhiPhi}, which goes as follows.
For $a\in \V$, denote by $\ell_a$ the corresponding \emph{linear} function on $\W$.
Consider the natural identifications $S_\infty\bb\simeq S_a^*\V$ and $S_b^*\W\simeq S_\infty\PP$, for $b\in\W$. Then, there are isomorphisms
\begin{equation}
\label{eq:introPhiGPsi}
\gnear[\infty]{\ell_a}(\plap K) \simeq \van^0(K), \quad
\van[b]^0(\plap K) \simeq \opb r\gnear[\infty]{-\ell_b}(K),
\end{equation}
where $r$ is the antipodal map.

Our proof of \eqref{eq:introPhiGPsi} proceeds as follows.
The second isomorphism is obtained from the first one by interchanging $\V$ and $\W$, and replacing $K$ by $\plap K$.
After translation from $a$ to $0$, the first isomorphism reads
\[
\gnear[\infty]0(\plap K) \simeq \van[0]^0(K).
\]
By definition, this is implied by the isomorphism
\[
\Enu^\sph_{\st \infty}(\plap K) \simeq \Emu^\sph_{\st 0}(K).
\]
We prove the above isomorphism using the so-called smash functor of \cite[\S6]{DHMS18}, in its enhanced version from \cite[\S6]{DK21a}.

\subsection{}
Concerning related literature, the $\D$-module counterpart
of \eqref{eq:introPhiGPsi} is proved in \cite{Mal91} when $\field=\C$ and $K=\dr^\enh(\shm)$ is the enhanced de Rham complex of a holonomic algebraic $\D_\V$-module $\shm$ which is regular at finite distance, and has only linear exponential factors at infinity. Note that, in this case, $\lap\shm$ satisfies the same conditions.

In the framework of enhanced ind-sheaves, a proof of \eqref{eq:introPhiGPsi} is given in \cite{DHMS18}, in the case where $K=e\iota(F)$, for $F$ a perverse sheaf\footnote{For $F=\dr(\shm)$, this means that $\shm$ is regular everywhere, including at infinity.} on $\V_\infty$. 

See \cite{Moc21} for a recent thorough treatment of the Fourier-Laplace transform of holonomic algebraic $\D$-modules on the affine line.

\subsection{}
The contents of this paper are as follows.

After recalling some notations in Section~\ref{se:not}, we recall in Section~\ref{se:perv} the notion of perverse enhanced ind-sheaf on a complex analytic curve. For such a perverse object, we show that its specialization and microlocalization are perverse sheaves in the classical sense.

In Section~\ref{se:psiphi} we discuss nearby and vanishing cycles along the lines presented in \S\ref{sse:intropsiphi} above.

In Section~\ref{se:Fou} we apply our constructions to the Fourier-Laplace transform in dimension one. In particular, we give a proof of \eqref{eq:introPhiGPsi}.

Finally, we present in the Appendix an alternative description of vanishing cycles in terms of blow-up transforms. In this setting, both nearby and vanishing cycles are realized on the circle of normal directions $S_a X$.

\section{Review on enhanced ind-sheaves}\label{se:not}

We recall here some notions and results, mainly to fix notations, referring to the literature for details. In particular, we refer to \cite{KS90} for sheaves, to \cite{Tam08} (see also \cite{GS14,DK16bis}) for enhanced sheaves, to \cite{KS01} for ind-sheaves, to \cite{DK16} (see also \cite{KS16,DK16bis,DK21b}) for bordered spaces and enhanced ind-sheaves, and to \cite{DK21a} for enhanced specialization and microlocalization.

\medskip

In this paper, $\field$ denotes a base field.

\subsection{Ind-sheaves and bordered spaces}

A good space is a topological space which is Hausdorff, locally compact, countable at infinity, and with finite soft dimension. 
Let $M$ be a good space.

Denote by $\Mod(\field_M)$ the category of sheaves of $\field$-vector spaces on $M$, and by $\BDC(\field_M)$ its bounded derived category. For $f\colon M\to N$ a morphism of good spaces, denote by $\tens$, $\opb f$, $\reim f$ and $\rhom$, $\roim f$, $\epb f$ the six operations.
 Denote by $\dual_M$ the Verdier dual.

For $S\subset M$ locally closed, we denote by $\field_S$ the extension by zero to $M$ of the constant sheaf on $S$ with stalk $\field$.

A bordered space is a pair $\bbM=(M,\ccM)$ with $M$ an open subset of a good space $\ccM$. We set $\obM\defeq M$ and $\cbM\defeq\ccM$.
A morphism $f\colon\bbM\to\bbN$ of bordered spaces is a morphism $\of\colon \obM\to \obN$ of good spaces such that the projection $\overline\Gamma_f\to \cbM$ is proper. Here, $\overline\Gamma_f$ denotes the closure in $\cbM\times\cbN$ of the graph $\Gamma_f$ of $\of$.
The morphism $f$ is called semi-proper if the projection $\overline\Gamma_f\to \cbN$ is proper. 

By definition, a subset $Z$ of $\bbM$ is a subset of $\obM$.
We say that $Z$ is relatively compact in $\bbM$ if it is contained in a compact subset of $\cbM$. 
For $Z\subset\bbM$ locally closed, we set $\bZ=(Z,\overline Z)$ where $\overline Z$ is the closure of $Z$ in $\cbM$. 

Let $\Mod_c(\field_\bbM)\subset\Mod(\field_\obM)$ be the full subcategory of sheaves on $\obM$ whose support is relatively compact in $\bbM$. We denote by $\Mod(\ifield_\bbM)$ the category of ind-sheaves on $\bbM$, that is the category of ind-objects with values in $\Mod_c(\field_\bbM)$. 
We denote by $\BDC(\ifield_\bbM)$ the bounded derived category of ind-sheaves of $\field$-vector spaces on $\bbM$, and by $\tens$, $\opb f$, $\reeim f$ and $\rihom$, $\roim f$, $\epb f$ the six operations.

We denote by $\iota_\bbM\colon\BDC(\field_\obM)\to\BDC(\ifield_\bbM)$ the natural embedding,
by $\alpha_\bbM$ its left adjoint, and we set $\rhom\defeq\alpha_\bbM\rihom$.
For $F\in \BDC(\field_\obM)$,
we often write simply $F$ instead of $\iota_\bbM F$
in order to make notations less heavy.

Recall that $\iota$ commutes with $\roim{\oo f}$, $\oo f{}^{-1}$ and $\oo f{}^!$, but it does not commute in general with $\reeim{\oo f}$.
If $f$ is semi-proper, then 
\begin{equation}\label{eq:semipr}
\iota_\bbN\,\reim{\oo f} \simeq \reeim f\,\iota_\bbM.
\end{equation}

\subsection{Enhanced ind-sheaves}

Denote by $t\in\R$ the coordinate on the affine line, consider the two-point compactification $\overline\R\defeq\R\cup\st{-\infty,+\infty}$, and set $\bR\defeq(\R,\overline\R)$.
For $\bbM$ a bordered space, consider the projection 
\[
\pi_\bbM\colon \bbM\times\bR\to \bbM.
\]

Denote by $\BEC \bbM\defeq\BDC(\ifield_{\bbM\times\bR})/\opb\pi_\bbM\BDC(\ifield_\bbM)$ the bounded derived category of enhanced ind-sheaves of $\field$-vector spaces on $\bbM$. Denote by $\quot_\bbM\colon\BDC(\ifield_{\bbM\times\bR})\to\BEC\bbM$ the quotient functor.
Recall that there is a natural splitting $\BEC\bbM\simeq\BECp\bbM\dsum\BECm\bbM$.

For $f\colon \bbM\to \bbN$ a morphism of bordered spaces, denote by $\ctens$, $\Eopb f$, $\Eeeim f$ and $\cihom$, $\Eoim f$, $\Eepb f$ the six operations.  
 Recall that the external operations are induced via $\quot$ by the corresponding operations for ind-sheaves with respect to the morphism $f_\R \defeq f \times \id_\bR$. 
Denote by $\Edual_\bbM$ the Verdier dual.
Denote by $\fhom$ the hom functor taking values in $\BDC(\field_\obM)$.

One sets
\[
\Efield_\bbM\defeq \quot_\bbM\bl \indlim[c\to+\infty]\field_{\st{t\geq c}} \br,
\]
writing for short $\st{t\geq c} = \st{(x,t)\in\obM\times\R\semicolon t\geq c}$.

There are embeddings
\begin{align*}
\epsilon^\pm_\bbM&\colon\BDC(\ifield_\bbM) \rightarrowtail \BECpm\bbM, &
F&\mapsto \quot_\bbM\field_{\st{\pm t\geq 0}}\tens\opb{\pi_\bbM} F, \\
e_\bbM&\colon\BDC(\ifield_\bbM) \rightarrowtail \BECp\bbM, &
F&\mapsto \Efield_\bbM\tens\opb{\pi_\bbM} F 
\simeq \Efield_\bbM\ctens\epsilon^+_\bbM(F).
\end{align*}
Recall that $e$ commutes with $\reeim f$, $\opb f$ and $\epb f$, but it does not commute in general with $\roim f$.

The functor $e_\bbM$ has as a left quasi-inverse the sheafification functor
\[
\sh_\bbM\colon\BECp\bbM\to\BDC(\field_\obM), \quad
K\mapsto\fhom(\Efield_\bbM,K).
\]
We call $\sh_\bbM(K)$ the sheaf associated with $K$.
We say that $K\in\BECp\bbM$ is of sheaf type if it is in the essential image of $e_\bbM\iota_\bbM$.
This is a local property\footnote{A property is local on $\bbM$ if any $x\in\cbM$ has an open neighborhood $V\subset\cbM$ such that the property holds on the associated bordered space $\inb{(V\cap\obM)}$} on $\bbM$.
The full subcategory of $\BECp\bbM$ consisting of objects
of sheaf type is closed by extensions,
and equivalent to $\BDC(\field_\obM)$.

For $U\subset\bbM$ an open subset, and $\varphi\colon\inb U\to\bR$ a morphism of bordered spaces, we set
\begin{equation}\label{eq:Ex}
\ex^\varphi_{U|\bbM} \defeq \quot_\bbM\field_{\st{t+\varphi(x)\geq0}}, \quad
\Ex^\varphi_{U|\bbM} \defeq \Efield_\bbM\ctens\ex^\varphi_{U|\bbM},
\end{equation}
writing for short $\st{t+\varphi(x)\geq0}=\st{(x,t)\in U\times\R\semicolon t+\varphi(x)\geq0}$.

\subsection{Specialization and microlocalization}\label{sse:Founumu}

Let $N$ be a smooth manifold, $V\to N$ an $\R$-vector bundle, and $\mathbb{S} V$ its fiberwise sphere compactification given by $\mathbb{S} V\defeq\bl(\R\times V)\setminus(\st0\times N)\br/\Gmp$.
Set $\inb V\defeq(V,\mathbb{S} V)$.
Let $V^*\to N$ be the dual bundle.

The enhanced Fourier-Sato transforms
\begin{align*}
(\ast)^\wedge &\colon \BECp {\inb V} \to \BECcon{\inb{V^*}}, \\
\lap{(\ast)} &\colon \BECp{\inb V}  \to \BECp{\inb{V^*}},
\end{align*}
are the integral transforms with kernel, respectively,
\[
\Fou \defeq \epsilon^+_{\inb{(V\times V^*)}}\field_{\st{\langle v,w\rangle\leq 0}}, \quad
\Lap \defeq \ex_{V\times V^*|\inb{(V\times V^*)}}^{-\langle v,w\rangle}
\]
Here\footnote{ What we denote here by $\BECcon{\inb{V^*}}$ corresponds to $\BECp{\inb{V^*}}\cap\BECcon{\inb{V^*}}$ in the notations of \cite{DK21a}}, $\BECcon{\inb{V^*}}$ is the full triangulated subcategory of $\BECp{\inb{V^*}}$ whose objects are conic for the natural action of the group object $\inb{(\Gmp)}$.  Recall that $\lap{(\ast)}$ and $(\ast)^\wedge$ agree on conic objects.

Let $M$ be a smooth manifold, $N\subset M$ a submanifold, and denote by
\[
\xymatrix{T_NM \ar[r]^-\tau & N & T^*_MN \ar[l]_-\varpi}
\]
the normal and conormal bundles.
Consider the normal deformation $\pnd\colon\ND M\to M$ with center $N$, and the associated commutative diagram of bordered spaces
\[
\xymatrix@R=3ex@C=8ex{
\inb{(T_NM)} \ar@{^(->}[r]^-{i_\nd} \ar[d]_\tau & \inb{(\ND M)} 
\ar[d]_{\pnd} \ar[r]^{\snd}&\inb{\R}\\
N \ar@{^(->}[r]_-{i_N} \ar@{}[ur]|-\square & M
&\inb\Omega \ar@{_(->}[lu]_-{j_\nd} \ar[l]^-{p_\Omega}
\ar[r]\ar@{}[u]|-\square 
&\inb{(\R_{>0})}\ar@{_(->}[lu],
}
\]
where $\inb{(\ND M)}$ is the bordered compactification of $\pnd$, and $\Omega\defeq\snd^{-1}(\R_{>0})$.
Sato's specialization and microlocalization functors have natural enhancements
\begin{align*}
\Enu_N &\colon \BECp N \to \BECcon{\inb{(T_NM)}}, \\
\Emu_N &\colon \BECp N \to \BECcon{\inb{(T_NM)}},
\end{align*}
defined by
\begin{align*}
\Enu_N(K) &\defeq \Eopb{i_\nd}\Eoim{{j_\nd}}\Eopb{p_\Omega}K, \\
\Emu_N(K) &\defeq \lap\Enu_N(K)\simeq\Enu_N(K)^\wedge.
\end{align*}

Denoting by $\bdot T_NM$ the complement of the zero-section, and setting $S_NM \defeq \bdot T_NM/\Gmp$, consider the natural morphisms
\[
\xymatrix{\inb{(T_NM)} & \inb{(\bdot T_NM)} \ar[l]_-u \ar[r]^-\gamma & S_NM}.
\]
We set
\[
\Enu_N^\sph\defeq\Eoim\gamma\Eopb u\Enu_N, 
\]
so that $\Eopb u\Enu_N\simeq\Eopb\gamma\Enu_N^\sph$.
We similarly define $\Emu_N^\sph$.

Consider the real oriented blowup $\prb\colon\RB M\to M$ with center $N$, and the associated commutative diagram of bordered spaces
\begin{equation}
\label{eq:blowMN}
\xymatrix@R=3ex{
S_NM \ar@{^(->}[r]^-{i_\rb} \ar[d]_\sigma & \RB M \ar[d]_{\prb}
& \inb{(M\setminus N)} \ar@{_(->}[dl]^-{j_N} \ar@{_(->}[l]_-{j_\rb} \\
N \ar@{^(->}[r]^-{i_N} \ar@{}[ur]|-\square & M.
}
\end{equation}
One has an associated functor
\[
\Enu_N^\rb \colon \BECp N \to \BECp{S_NM}, \quad 
K\mapsto\Eopb i_\rb\Eoim{{j_\rb}}\Eopb j_N K.
\]
Note that one has
\begin{equation}\label{eq:sphrb}
\Enu_N^\sph\simeq\Enu_N^\rb.
\end{equation}

\subsection{Constructibility}

Let $\bbM$ be a subanalytic bordered space.

We denote by $\BDC_\Rc(\field_\bbM)$ the full triangulated subcategory of $\BDC(\field_\obM)$ whose objects $F$ are such that $\reim{{j_\bbM}}F$ is $\R$-constructible in $\cbM$.
Here, $j_\bbM\colon\bbM\hookrightarrow\cbM$ is the natural morphism.

We denote by $\BECRc\bbM$ the full triangulated subcategory of $\BECp \bbM$ whose objects $K$ satisfy the following property. For any open relatively compact subanalytic subset $U\subset\bbM$ there exists $F\in\BDC_\Rc(\field_{\bbM\times\bR})$ such that $\opb\pi\field_U\tens K \simeq \Efield_\bbM\ctens \quot_\bbM\, F$.

\section{Enhanced perverse ind-sheaves on a curve}\label{se:perv}

In this section we let $X$ be a smooth complex curve.

\subsection{Normal form}

Consider the real oriented blow-up $\RB[a] X$ of $X$ with center $a\in X$ as in \eqref{eq:blowMN},
and the associated natural morphisms
\begin{equation}
\label{eq:blowonline}
\xymatrix@R=3ex{
S_a X \ar@{^(->}[r]^-i & \RB[a] X
& \inb{(X\setminus\{a\})} \ar@{^(->}[r]^-{j_a} \ar@{_(->}[l]_-j & X\,,
}
\end{equation}
where we write for short $i=i_\rb$, $j=j_\rb$, and $j_a=j_{\st a}$.

A \emph{sectorial neighborhood} of $\theta\in S_a X$ is an open subset $U\subset X\setminus\st a$ such that $S_a X\cup  j (U)$ is a neighborhood of $\theta$ in $\RB[a] X$. We write $U \dotowns \theta$ to indicate that $U$ is a sectorial neighborhood of $\theta$.
We say that $U\subset X\setminus \st a$ is a sectorial neighborhood of $Z\subset S_a X$, and we write $U\dotsupset Z$, if $U$ is a sectorial neighborhood of each $\theta\in Z$.

The sheaf $\shp_{S_a X}$ of \emph{Puiseux germs} on $S_a X$ is the subsheaf of $\opb i\oim j\opb j_a\O_X$ whose stalk at $\theta\in S_aM$ are holomorphic functions on small sectors
$V\dotowns\theta$ admitting a Puiseux expansion at $a$. 
We denote by $\oshp_{S_a X}$ the quotient of $\shp_{S_a X}$ modulo bounded functions, and we denote by $[f]\in\oshp_{S_a X}$ the equivalence class of $f\in\shp_{S_a X}$.

For $f\neq0$, we set $\ord_a(f)=-n_0/m$ if $f$ has a Puiseux expansion $\sum_{n\geq n_0}c_n z_a^{n/m}$ with $c_{n_0}\neq 0$, where $n,n_0\in\Z$, $m\in\Z_{>0}$, and $z_a$ is a local coordinate at $a$ with $z_a(a)=0$. We set $\ord_a(0)=-\infty$.
Note that $f$ is bounded if and only if $\ord_a(f)\leq 0$. 

\begin{definition}
One says that $K\in\BECRc X$ has \emph{normal form} at $\theta\in S_a X$ if there exist a finite subset $\Phi_\theta\subset\shp_{S_a X,\theta}$ and integers $n_\theta(f)\in\Z_{> 0}$ for $f\in\Phi_\theta$ such that
\begin{equation}
\label{eq:Knormal}
\opb\pi\field_{V_\theta}\tens K \simeq \DSum_{f\in\Phi_\theta} \bl\Ex^{\Re f}_{V_\theta|X}\br^{n_\theta(f)}[1]
\end{equation}
for some $V_\theta\dotowns\theta$. (Recall that $\Ex^{\Re f}_{V_\theta|X}$ was defined in \eqref{eq:Ex}.)
One says that $K$ has normal form at $I\subset S_a X$ if it has normal form at any $\theta\in I$. One says that $K$ has normal form at $a$ if it has normal form at $S_a X$.
\end{definition}

If $K$ has normal form at a connected open subset $I\subset S_a X$, and $f\in\shp_{S_a X}(I)$, the number
\[
\ol N([f]) = \sum_{h\in\Phi_\theta,\ [f]=[h]}n_\theta(h)
\]
is finite, does not depend on the choice of $\theta\in I$, and only depends on the class $[f]$ of $f$.
If $\ol N([f])>0$, one says that $[f]$ is an \emph{exponential factor} of $K$, and $\ol N([f])$ 
is called its \emph{multiplicity}.

\begin{proposition}\label{pro:nuSa}
Let $K\in\BECRc X$ have normal form at $I\subset S_a X$. Then $\Enu^\rb_{\st a}(K)|_I$ is of sheaf type. More precisely, $\Enu^\rb_{\st a}(K)|_I\simeq e\iota(L)$ for $L\in\Mod(\field_I)$ a local system of rank $\ol N([0])$.
\end{proposition}

\begin{proof} 
The statement is a local problem on $I$.

Let $\theta\in I$. Since $K$ has normal form at $\theta$, there is an open neighborhood $I_\theta\owns\theta$ such that \eqref{eq:Knormal} holds with $V_\theta\dotsupset I_\theta$.
Thus, we can reduce to the case $K\simeq\Ex_{V_\theta|X}^{\Re h}[1]$ for $h\in\shp_{S_a X}(I_\theta)$. 
By definition of $\Enu^\rb_{\st a}$, it is then enough to check that
\[
\left.\Eopb i\Eoim j\Eopb j_a\Ex_{V_\theta|X}^{\Re h}\right|_{I_\theta} \simeq
\begin{cases}
e\iota(\field_{I_\theta}) &\text{if $\ord_a(h)\le0$,} \\
0 &\text{if }\ord_a(h)>0.
\end{cases}
\]
The statement is clear if $\ord_a(h)\le0$. If $\ord_a(h)>0$, after a change of variable and a ramification we can assume that $h(z)=z_a^{-1}$.
Hence, one concludes using Lemma~\ref{lem:Exy}.
\end{proof}

\Lemma\label{lem:Exy}
Let $M=\R_{\ge0}\times\R$ with coordinates $(\rho,s)$. Set $U=\st{\rho>0}$, 
$N=\st{\rho=0}$, and 
consider the embeddings
$N \To[i]  M \xleftarrow{\ j\ } \inb U$. Then, one has
\eq
&&\ex_{U|M}^{s/\rho} \simeq \Eoim j\Eopb j \ex_{U|M}^{s/\rho},\label{eq:UM1}\\
&&\Eopb i \Eoim j \Eopb j \ex_{U|M}^{s/\rho} \simeq 0,\label{eq:UM2}\\
&&\sh_M\bl \Ex_{U|M}^{s/\rho}\br \simeq\cor_{\st{\rho>0}\cup\st{s<0}}.
\label{eq:UM3}
\eneq
\enlemma

\begin{proof}
 (i) One has
\begin{align*}
\ex_{U|M}^{s/\rho} 
&= \quot_M\field_{\{\rho>0,\ t+s/\rho\geq0\}} \\
&= \quot_M\field_{\{\rho>0,\ s+\rho t\geq0\}}, \\
\Eoim j \Eopb j \ex_{U|M}^{s/\rho}
&\simeq \quot_M \roim {{j_\R}} \opb j_\R\field_{\{\rho>0,\ s+\rho t\geq 0\}} \\
&\simeq \quot_M \field_{\{\rho\geq0,\ s+\rho t\geq 0\}} .
\end{align*}
Consider the distinguished triangle
\[
\field_{\{\rho>0,\ s+\rho t\geq 0\}} \to
\field_{\{\rho\geq0,\ s+\rho t\geq 0\}} \to
\field_{\{\rho=0,\ s\geq 0\}} \to[+1].
\]
Since $\quot_M \field_{\{\rho=0,\ s\geq 0\}} \simeq 0$, \eqref{eq:UM1} follows.

\smallskip\noindent(ii)
\eqref{eq:UM2} is implied by \eqref{eq:UM1}, since $\Eopb i\ex_{U|M}^{s/\rho} \simeq \quot_N\opb i_\R\field_{\{\rho>0,\ s+\rho t\geq0\}} \simeq 0$.

\smallskip\noindent(iii)
By \cite[Corollary 3.7]{DK21b}, denoting by $\LE$ the left adjoint to $\quot$, one has
\begin{align*}
\sh_M\bl \Ex_{U|M}^{s/\rho}\br 
&\simeq \roim\pi \LE \ex_{U|M}^{s/\rho} \\
&\simeq \roim\pi \field_{\{\rho>0,\ s+\rho t\geq0\}} \\
&\simeq \dual_M\reim\pi \dual_{M\times\R}\field_{\{\rho>0,\ s+\rho t\geq0\}}\\
&\simeq \bl\dual_M\reim\pi \field_{\{\rho\geq0,\ s+\rho t>0\}}\br[-3] \\
&\simeq \bl\dual_M \field_{\{\rho>0\}\cup\{s>0\}}\br[-2] \\
&\simeq \field_{\{\rho>0\}\cup\{s<0\}}.
\end{align*}

\end{proof}

\subsection{Perversity}
 For $a\in X$, let $i_a\colon \{a\} \to X$ be the embedding. 
Recall that $F\in\BDC_\Rc(\field_X)$ is perverse if and only if there exists a discrete subset $\Sigma\subset X$ such that:
\begin{itemize}
\item[(a)]
$H^n \opb i_a F = 0$ for any $n>0$ and $a\in\Sigma$;
\item[(b)]
$H^n \epb{i_a} F = 0$ for any $n<0$ and $a\in\Sigma$;
\item[(c)]
$F|_{X\setminus\Sigma}\simeq L[1]$, for $L\in\Mod(\field_{X\setminus\Sigma})$ a local system of finite rank.
\end{itemize}
Denote by $\Perv(\field_X)\subset\BDC_\Rc(\field_X)$ the full triangulated subcategory of perverse sheaves.

\begin{definition}\label{def:ECc}
\bnum
\item
We say that $K\in\BECRc X$ is $\C$-constructible
if there exists a discrete subset $\Sigma\subset X$ such that:
\bna
\item for any $n\in\Z$,
$H^n(K)|_{X\setminus\Sigma}\simeq e\iota( L_n)$ for a local system $ L_n$ on $X\setminus\Sigma$
of finite rank,
\item for any $n\in\Z$,
$H^{n}(K)$ has normal form at any $a\in\Sigma$.
\ee
Denote by $\BECCc X\subset\BECRc X$ the full triangulated subcategory of 
$\C$-constructible enhanced ind-sheaves.

\item
We say that $K\in\BECRc X$ is an enhanced perverse ind-sheaf if there exists a discrete subset $\Sigma\subset X$ such that:
\bna
\item
$H^n \Eopb i_a K = 0$ for any $n>0$ and $a\in\Sigma$;
\item
$H^n \Eepb i_a K = 0$ for any $n<0$ and $a\in\Sigma$;
\item
$K|_{X\setminus\Sigma}\simeq e\iota(L[1])$, for $L\in\Mod(\field_{X\setminus\Sigma})$ a local system of finite rank,
\item
$H^{-1}(K)$ has normal form at any $a\in\Sigma$.
\ee
Denote by $\Eperv(\ifield_X)\subset\BECCc X$ the full subcategory of enhanced perverse ind-sheaves.
\ee
\end{definition}

\Lemma 
The functor 
$e\,\iota\cl \BDC(\cor_X)\to \BECp X$ sends $\Perv(\cor_X)$ to  $\Eperv(\ifield_X)$,
and the functor 
$\sh\cl \BECp X\to \BDC(\cor_X)$ sends $\Eperv(\ifield_X)$ to $\Perv(\cor_X)$.
\enlemma

\begin{proof}
The first statement is clear from the definitions. The second statement follows using Lemma~\ref{lem:Exy}.
\end{proof}

Note that $\Eperv(\ifield_X)$ is an abelian subcategory of the quasi-abelian heart ${}^{1/2}\enh^0_\Rc(\ifield_X)$ for the middle perversity $t$-structure introduced in \cite{DK16bis}. 
Note also that, using \cite[Proposition 4.1.2]{DK18} (see also \cite[Proposition 3.28]{Moc16}), one has

\begin{theorem}
The enhanced de Rham functor induces an equivalence between $\Eperv(\ind\C_X)$ and the category of holonomic $\D_X$-modules.
\end{theorem}

\begin{proposition}\label{pro:numuperv}
Let $K\in\Eperv(\ifield_X)$ and $a\in X$ a singularity of $K$. Then both $\Enu_{\st a}(K)$ and $\Emu_{\st a}(K)[1]$ are of sheaf type. Moreover they, as well as their associated sheaves, are perverse with the zero-section as their only singularity.
\end{proposition}

\begin{proof}
Consider the morphisms
\[
\xymatrix{
X & \ar@{_(->}[l]_{i_a} \st a \ar@{^(->}[r]^-o & T_a X & \ar@{_(->}[l]_-u \inb{(\bdot T_a X)} \ar[r]^\gamma & S_a X,
}
\]
and consider the distinguished triangle
\begin{equation}
\label{eq:conEnu1}
\Eeeim u\Eopb{u}\Enu_{\st{a}} K \to \Enu_{\st{a}} K \to \Eoim o \Eopb{o}\Enu_{\st{a}} K
 \To[+1].
\end{equation}

\smallskip\noindent(i)
Let us show that $\Enu_{\st a} K$ is of sheaf type.
 By  Proposition~\ref{pro:nuSa}, there exists a local system 
$L\in\Mod(\field_{S_a X})$ such that $\Enu^\sph_{\st a}(K)\simeq e\iota(L[1])$.
Since $u$ is semiproper, one has by \eqref{eq:semipr}
\[
\Eeeim u\Eopb{u}\Enu_{\st a} K \simeq\Eeeim u\Eopb \gamma\Enu^\sph_{\st a} K \simeq e\iota(\reim u\opb \gamma L[1]).
\]
Hence $\Eeeim u\Eopb{u}\Enu_{\st a} K $ is of sheaf type.
Since any $\R$-constructible enhanced ind-sheaf on a point is of sheaf type,
$\Eopb{o}\Enu_{\st{a}} K$ is of sheaf type.
This implies that $\Enu_{\st a} K$ is of sheaf type by
the distinguished triangle \eqref{eq:conEnu1}.

\smallskip\noindent(ii)
Let us show that $\Enu_{\st a}(K)\in\BECRc{T_a X}$ is perverse, with $\st a$ as its only singularity.
Since
$\Eopb{o}\Enu_{\st a}(K)\simeq\Eopb{i_a}K$
and $\Eepb{o}\Enu_{\st a}(K)\simeq\Eepb{i_a}K$ by \cite[Lemma~4.8]{DK21a},
we have
\begin{itemize}
\item[(a)] $H^n(\Eopb{o}\Enu_{\st a}(K))\simeq
H^n(\Eopb{i_a}K)\simeq 0$ for $n>0$,
\item[(b)] 
$H^n(\Eepb{o}\Enu_{\st a}(K))\simeq
H^n(\Eepb{i_a}K)\simeq 0$ for $n<0$.
\item[(c)] $\Eopb u\Enu_{\st a}(K)\simeq e\iota(\opb\gamma L[1])$, with $L$ as in (i).
\end{itemize}

\smallskip\noindent(iii)
It remains to show that $\Emu_{\st a}(K)[1]$ is of sheaf type, and that its associated sheaf is perverse. Setting $F\defeq\sh\bl\Enu_{\st a}(K)\br$, this follows from
\[
\Emu_{\st a}(K) \simeq \Enu_{\st a}(K)^\wedge \simeq e\iota(F)^\wedge \simeq e\iota(F^\wedge),
\]
and the fact that the classical Fourier-Sato transform for sheaves preserves the perversity of $\Gmp$-conic objects, up to shift $[1]$.
\end{proof}

\section{Nearby and vanishing cycles}\label{se:psiphi}

As we mentioned in the Introduction, nearby cycles for enhanced ind-sheaves were already discussed in \cite[\S6.2]{DK18}. However, defining them through enhanced specialization, as we do here, sheds some light on the underlying geometry. Moreover, using enhanced microlocalization, we can here also deal with vanishing cycles. In this section we thus recall and complement some results from loc.\ cit.

\subsection{Definitions}\label{sse:psiphi}
Let $X$ be a smooth complex curve, and $a\in X$.
Consider the natural morphisms associated with the real blow-up $\RB[a]X$ of $X$ with center $a$ as in \eqref{eq:blowonline}:
\[
\xymatrix@R=3ex{
S_a X \ar@{^(->}[r]^-i & \RB[a] X
& \inb{(X\setminus\{a\})} \ar@{^(->}[r]^-{j_a} \ar@{_(->}[l]_-j & X\,.
}
\]

\begin{definition}\label{def:psiphi}
Let $K\in\BECRc X$.
\begin{itemize}
\item[(i)]
Consider the objects of $\BDC(\field_{S_a X})$
\begin{align*}
\near(K) 
&\defeq \nu_{\st a}^\sph\bl\sh_X(K)\br \simeq \nu_{\st a}^\rb\bl\sh_X(K)\br \\
&= \opb i \roim j \opb j_a\sh_X (K) \\
&\underset{(*)}\simeq \opb i \roim j \sh_{\inb{(X\setminus\{a\})}} \bl\Eopb j_a K\br, \\
\fnear 0 (K) 
&\defeq \opb i \sh_{\RB[a] X} \bl\Eoim j \Eopb j_a K\br, \\
\gnear 0 (K) 
&\defeq \sh_{S_a X} \bl\Eopb i \Eoim j \Eopb j_a K\br \\
&= \sh_{S_a X} \bl \Enu_{\st a}^\rb(K)\br \simeq \sh_{S_a X} \bl \Enu_{\st a}^\sph(K)\br,
\end{align*}
where $(*)$ follows from \cite[Lemma~3.9]{DK21b}.
\item[(ii)]
Let $I\subset S_a X$ be an open subset and $f\in\shp_{S_a X}(I)$.
For $U\subset X\setminus\st a$ an open subset such that $U\dotsupset I$
and $f$ extends on $U$, set
\begin{align*}
K(f) &\defeq \cihom(\Ex^{\Re f}_{U|X},K)\in\BECp X , \\
\fnear f (K) 
&\defeq \fnear 0\bl K(f)\br|_I\in\BDC(\field_{I}), \\
\gnear f (K) 
&\defeq \gnear 0\bl K(f)\br|_I\in\BDC(\field_{I}).
\end{align*}
Note that $\fnear f(K)$ and $\gnear f(K)$ do not depend on the choice of $U$.
\item[(iii)]
Consider the object of $\BDC(\field_{S_a^*X})$
\begin{align*}
\van(K) &\defeq \mu_{\st a}^\sph\bl\sh_X(K)\br[1] , \\
\van^0(K) &\defeq \sh_{S_a^*X}\bl\Emu_{\st a}^\sph(K)\br[1] .
\end{align*}
\end{itemize}
\end{definition}

\begin{lemma}
Let $I\subset S_a X$ be an open subset, $f,g\in\shp_{S_a X}(I)$ with $f\eprec_I g$, and $K\in\BECp X$. 
Then there are natural morphisms in $\BECp I$
\[
\near(K)|_I \from \fnear g(K) \from \fnear f(K) \to \gnear f(K). 
\]
\end{lemma}

\begin{proof}
It follows from \cite[Lemma~3.9]{DK21b}.
\end{proof}

Let $\theta\in S_a X$, $f\in\shp_{S_a X,\theta}$,
and denote by $z_a$ a local coordinate at $a$ with $z_a(a)=0$.
Assuming  $K\in\BECRc X$, it follows from \cite[Lemma~5.1]{DK21b} that one has
\begin{equation}\label{eq:germ}
\gnear f(K)_\theta \simeq
\ilim[\delta,\varepsilon\to0+,\ V\dotowns \theta]\FHom(\exx^{\range{\Re f(x)}{\Re f(x) -\delta|z_a(x)|^{-\varepsilon}}}_{V|X}, K).
\end{equation}

\subsection{The case of perverse objects}

Let us collect in the following lemma some results from \cite[\S6]{DK18}. Note that statement (iii) below also follows from Proposition~\ref{pro:nuSa}.

\begin{lemma}\label{lem:psiphinorm}
 Let $I\subset S_aX$ be a connected open subset. 
Let $f,g\in\shp_{S_a X}(I)$ with $f\eprec_I g$, and $K\in\BECRc X$. 
Assume that $K$ has normal form at $I$. Then
\begin{itemize}
\item[(i)]
$\near(K)|_I$ is concentrated in degree zero and is a local system on $I$ of rank $\sum_{[h]\in\oshp_{S_a X,\theta}}\overline N(h)$  for  $\theta\in I$,
\item[(ii)] 
$\fnear f(K)$ is concentrated in degree zero, and is an $\R$-constructible sheaf on $I$. Moreover, the morphisms $\fnear f(K) \to \fnear g(K) \to \near(K)|_I$ are monomorphisms, and $\fnear f(K) \to \gnear f(K)$ is an epimorphism.
\item[(iii)] 
$\gnear f (K)$ is concentrated in degree zero, and is a local system of rank $\overline N([f])$ on $I$.
\end{itemize}
\end{lemma}

Recall the notion of a Stokes filtration, e.g.\ from \cite[\S6.1]{DK18}.

\begin{proposition}
Let $K\in\Eperv(\ifield_X)$. Then
\begin{itemize}
\item[(i)]
$\near(K)$ is a local system on $S_a X$ with Stokes filtration $\fnear \bullet(K)$, and associated graded components $\gnear \bullet(K)$ which is a local system on $S_aX$;
\item[(ii)] 
$\van^0(K)$  is a local system on $S^*_a X$.
\end{itemize}
\end{proposition}

\begin{proof}
(i) is a particular case of Lemma~\ref{lem:psiphinorm}, and
(ii) follows from Proposition~\ref{pro:numuperv}.
\end{proof}

Refer to Appendix~\ref{se:lambda} for an alternative description of the vanishing cycles $\van^0(K)$ as a local system on $S_a X$, via some blow-up transforms.

\section{Fourier transform on the affine line}\label{se:Fou}

Let $K$ be an enhanced perverse ind-sheaf on the affine line, and assume that so is its shifted enhanced Fourier-Sato transform $\plap K\defeq\lap K[1]$.
The stationary phase formula provides a relation (see \eqref{eq:introPhiPhi}) between the graded components of the Stokes filtrations of $K$ and $\plap K$, for degrees which are not linear (modulo bounded function). We discuss here the case of linear degrees.

\subsection{Linear exponential factors}

Let $z$ be a coordinate on the complex line $\V$, and $w$ the dual coordinate on  $\W$, so that the pairing $\V\times \W\to\C$ is given by $(z,w)\mapsto z w$.
The underlying real vector spaces are in duality by the pairing $\langle z,w\rangle=\Re(z w)$.
Denoting by $\PP=\V\cup\st{\infty}$ the complex projective line with affine chart $\V$, one has $\V_\infty\simeq(\V,\PP)$. Similarly, $\W_\infty\simeq(\W,\bb)$, for $\bb=\W\cup\st{\infty}$.

Let $\Eperv(\ifield_{\V_\infty})$ be the full triangulated subcategory of $\BECRc{V_\infty}$ whose objects are of the form $\Eopb j K$ for some $K\in\Eperv(\ifield_{\PP})$.
Here, $j\colon\V_\infty\hookrightarrow\PP$ is the natural morphism. 

Consider the enhanced Fourier-Sato transform
\[
\BECp{\V_\infty}\to\BECp{\W_\infty}, \quad
K\mapsto\plap K\defeq\lap K[1].
\]
Here, the shift ensures compatibility with the Riemann-Hilbert correspondence.

\begin{theorem}\label{thm:EFouPerv}
Let $K\in\Eperv(\ifield_{\V_\infty})$. Assume  $\plap K\in\Eperv(\ifield_{\W_\infty})$. Then:
\begin{itemize}
\item[(i)]
for any $a\in\V$, under the canonical identification $S_\infty\bb \simeq S_a^*\V$, there is an
isomorphism of local systems
\[
\gnear[\infty]{aw}(\plap K) \simeq \van^0(K)\,;
\]
\item[(ii)]
for any $b\in\W$, under the canonical identification $S_b^*\W \simeq S_\infty\PP$, there is an
isomorphism of local systems
\[
\van[b]^0(\plap K) \simeq \opb{r}\gnear[\infty]{-bz}(K)\,,
\]
\end{itemize}
where $r$ denotes the antipodal map.
\end{theorem}

\begin{remark}
With notations as in \S\ref{sse:RH}, for $\field=\C$ let $K\defeq\dr^\enh(\shm)$ for $\shm$ an algebraic holonomic $\D_\V$-module. Then $\plap K \simeq \dr^\enh(\shm^\wedge)$, where $\shm^\wedge$ is the Fourier-Laplace transform of $\shm$. Since $\shm^\wedge$ is an algebraic holonomic $\D_{\W}$-module, $\plap K\in\Eperv(\ifield_{\W_\infty})$.
\end{remark}

\begin{proof}[Proof of Theorem~\ref{thm:EFouPerv}]
(ii) follows from (i).
In fact, interchanging the roles of $\V$ and $\W$, one has
\[
\van[b]^0(\plap K)
\simeq \gnear[\infty]{bz}(\plap\plap K) 
\underset{(*)}\simeq \gnear[\infty]{bz}(\Eopb r K) 
\simeq \opb r \gnear[\infty]{-bz}(K).
\]
For $(*)$ refer e.g.~to \cite[\S5.2]{DK21a}.

\smallskip\noindent(i)
The translation $\tau_a\colon\V\to\V$, $z\mapsto z+a$,
induces an identification $S_a\V\simeq S_0\V$. Moreover, one has
\begin{align*}
\van^0(K) &\simeq \van[0]^0(\Eopb{\tau_a}K), \\
\gnear[\infty]{aw}\bl\plap(\Eopb{\tau_a}K)\br 
&\simeq \gnear[\infty]{aw}\bl(\plap K)(-aw)\br \\
&\simeq \gnear[\infty]0\bl\plap K\br .
\end{align*}
Hence, we may assume $a=0$.
It is then enough to check that there is an isomorphism
\[
\gnear[\infty]0(\plap K) \simeq \van[0]^0(K).
\]
Since $\Enu^\sph_{\st{\infty}}(\plap K)$ and $\Emu_{\st 0}^\sph(K)$ are of sheaf type, it is equivalent to prove that there is an isomorphism
\[
\Enu^\sph_{\st{\infty}}(\plap K) \simeq \Emu_{\st 0}^\sph(K).
\]
This follows from Lemma~\ref{lem:nusm2} and \eqref{eq:smmu} below.
\end{proof}

\subsection{Smash functor}
We consider here the smash functor of \cite{DHMS18}, in its enhanced version from \cite{DK21a}, and establish a small additional result needed to complete the proof of Theorem~\ref{thm:EFouPerv}.

\smallskip
The sphere compactification $\SV\defeq\bl(\R_u\times \V)\setminus\st{(0,0)}\br/\Gmp$ of $\V$ decomposes as
$\SV = \V^+ \dunion H \dunion \V^-$, corresponding to $u>0$, $u=0$ or $u<0$.
Let us identify $\V=\V^+$.
Note that $H$ is a real hypersurface of $\SV$.
One has a natural identification 
$\bdot \V\seteq\V\setminus\st{0} = T^+_H\SV$, where $T^+_H\SV\subset\bdot T_H\SV$ denotes the normal directions pointing to $\V=\V^+$.
With these identifications, $\Enu_H$ induces a functor
\[
\Enu_{H|\bdot \V} \colon \BECp{\inb{\bdot\V}} \to \BECcon{\inb{\bdot\V}}
\]
which can be considered a ``specialization at $\infty$''.

The enhanced smash functor
\[
\Esm_\V\colon \BECp{\inb\V} \to \BECcon{\inb\V},
\]
for which we refer to \cite[\S6]{DK21a}, provides an extension of $\Enu_{H|\bdot \V}$ from $\inb{\bdot\V}$ to $\V$. In fact, $\Esm_\V$ induces a functor
\[
\Esm_{\V|\bdot \V} \colon \BECp{\inb{\bdot\V}} \to \BECcon{\inb{\bdot\V}},
\]
and one has
\begin{equation}\label{eq:smnu}
\Esm_{\V|\bdot \V} \simeq \Enu_{H|\bdot \V}.
\end{equation}
Recall also that, by \cite[Proposition~6.6]{DK21a}, for $K\in\BECp{\inb\V}$ one has
\begin{equation}\label{eq:smmu}
\Emu_{\st0}(K) \simeq \Esm_{\W}(\lap K).
\end{equation}

\begin{lemma}\label{lem:nusm2}
Let $K\in\BECp{\inb\V}$. Then, with the natural identification $S_\infty\PP\simeq\bdot\V/\Gmp$, one has
\[
\Enu^\sph_{\st\infty}(K) \simeq \Esm_\V^\sph(K).
\]
\end{lemma}

\begin{proof}
With $H,\V^\pm\subset\SV$ defined as above, one has $\RB[H]\SV = (\RB[H]\SV)^+ \dunion (\RB[H]\SV)^+$ and $S_H\SV=(S_H\SV)^+\dunion(S_H\SV)^-$.
Moreover, the identification $\V\simeq\V^+$ implies identifications $\RB[\infty]\PP\simeq(\RB[H]\SV)^+$ and $S_\infty\PP\simeq (S_H\SV)^+$. 

With these identifications, and using \eqref{eq:sphrb} and \eqref{eq:smnu}, it is enough to prove the isomorphism
\[
\Enu^\rb_{\st\infty}(K) \simeq \Enu_H^\rb(K)|_{S^+_H\SV}.
\]
This follows by considering the commutative diagram
\[
\xymatrix{
\inb\V \ar@{-}[d]^\wr & \inb{\bdot\V} \ar@{_(->}[l] \ar@{^(->}[r]  \ar@{-}[d]^\wr  & \RB[\infty]\PP \ar@{-}[d]^\wr  & S_\infty\PP \ar@{_(->}[l] \ar@{-}[d]^\wr \\
\inb{\V^+} & \inb{\bdot\V^+} \ar@{_(->}[l] \ar@{^(->}[r] \ar@{}[ur]|-\square & \bl\RB[H]\SV\br^+ \ar@{}[ur]|-\square & \bl S_H\SV\br^+. \ar@{_(->}[l]
}
\]
\end{proof}

\appendix

\section{Vanishing cycles by blow-up transform}\label{se:lambda}

\subsection{Blow-up transforms}
Let $M$ be a real analytic manifold, and $N\subset M$ a smooth submanifold.
As in \eqref{eq:blowMN} consider the real oriented blowup $\RB M$ of $M$ with center $N$, and the associated commutative diagram of bordered spaces
\[
\xymatrix@R=3ex{
S_NM \ar@{^(->}[r]^-{i} \ar[d]_\sigma & \RB M \ar[d]_{p}
& \inb{(M\setminus N)} \ar@{_(->}[dl]^-{j_N} \ar@{_(->}[l]_-{j} \\
N \ar@{^(->}[r]^-{i_N} \ar@{}[ur]|-\square & M,
}
\]
where we write for short $i=i_\rb$, $j=j_\rb$ and $p=\prb$.

\begin{definition}
For $K\in\BECp M$, consider the objects of $\BECp{S_NM}$
\begin{align*}
\enh\lambda^\rb_N(K)
&\defeq \Eepb i\Eopb p K[1], \\
\enh\widetilde\lambda^\rb_N(K)
&\defeq \Eopb i\Eepb p K.
\end{align*}
We denote by $\lambda^\rb_N$ and $\widetilde\lambda^\rb_N$ the analogous functors for sheaves.
\end{definition}

Note that one has
\[
e \circ \lambda^\rb_N \simeq \enh\lambda^\rb_N \circ e, \quad
e \circ \widetilde\lambda^\rb_N \simeq \enh\widetilde\lambda^\rb_N \circ e,
\]
and similarly for $e$ replaced by $\epsilon$, $\epsilon^+$ or $\epsilon^-$.

\begin{remark}
Note that $\enh\lambda^\rb_N\not\simeq\enh\widetilde\lambda^\rb_N$ in general, as shown by the following example. (See however Proposition~\ref{pro:phinu'S}.)
For $M=\R_x$ and $N=\{0\}$, one has $\RB M\simeq \st{x\leq 0}\dunion\st{x\geq 0}$. Restricted to the left component, the maps $i$ and $p$ are the embeddings $\st{0}\to[i]\st{x\leq 0}\to[p]\R$. Then, for $F=\field_{\st{x>0}}$, one has
\begin{align*}
\lambda^\rb_N(F) &\simeq \epb i\opb p \field_{\st{x>0}}[1] \simeq \epb i\bl\field_{\st{x>0}}|_{\st{x\leq 0}}[1]\br \simeq 0,\\
\widetilde\lambda^\rb_N(F) &\simeq \opb i\epb p \field_{\st{x>0}} \simeq \bl\rsect_{\st{x\leq 0}}\field_{\st{x>0}}\br_0 \simeq \field[-1].
\end{align*}
\end{remark}

\begin{lemma}\label{lem:dtpsiphi}
For $K\in\BECp M$, there are distinguished triangles
\begin{itemize}
\item[(i)] $\Eopb\sigma\Eopb i_N K \to\Enu^\rb_N(K) \to[c] \enh\lambda^\rb_N(K) \to [+1]$,
\item[(ii)] $\Eepb\sigma\Eepb i_N K \to\enh\widetilde\lambda^\rb_N(K) \to[v] \Enu^\rb_N(K) \to [+1]$.
\end{itemize}
\end{lemma}

\begin{proof}
(i) For $L\in\BECp{\RB M}$, there is a distinguished triangle
\[
\Eeeim j\Eopb j L \to
L \to
\Eeeim i\Eopb i L \to[+1].
\]
When $L=\Eopb p K$, the above distinguished triangle reads
\[
\Eeeim j\Eepb j_N K \to \Eopb p K \to \Eeeim i\Eopb\sigma \Eopb i_N K\to[+1].
\]
By applying $\Eepb i$ we get (i).

\smallskip\noindent(ii)
Consider the distinguished triangle
\[
\Eoim i\Eepb i L \to
L \to
\Eoim j\Eepb j L \to[+1].
\]
When $L=\Eepb p K$, the above distinguished triangle reads
\[
\Eoim i\Eepb \sigma \Eepb{i_N} K \to
\Eepb p K \to
\Eoim j\Eopb {j_N} K \to[+1].
\]
One concludes by applying $\Eopb i$.
\end{proof}

The following result is clear from the definitions and \cite[Lemma~4.7]{DK21a}.

\begin{lemma}\label{lem:rbdual}
For $K\in\BECRc M$, one has
\[
\Edual\Enu^\rb_N(K) \simeq \Enu^\rb_N(\Edual K), \quad
\Edual\enh\widetilde\lambda^\rb_N(K) \simeq \enh\lambda^\rb_N(\Edual K)[-1].
\]
\end{lemma}

\begin{lemma}\label{lem:KnuK2}
For $K\in\BECp M$ one has
\begin{itemize}
\item[(i)] $\enh\lambda^\rb_N(K) \simeq \enh\lambda^\rb_N(\Enu_N(K))$,
\item[(ii)] $\enh\widetilde\lambda^\rb_N(K) \simeq \enh\widetilde\lambda^\rb_N(\Enu_N(K))$,
\end{itemize}
with the identification $S_NM\simeq S_N(T_NM)$.
\end{lemma}

\begin{proof}
Since the proofs are similar, we will only discuss (i).

\smallskip\noindent(a)
We will construct in part (b) below a natural morphism
\begin{equation}
\label{eq:lambdanu}
\enh\lambda^\rb_N(\Enu_N(K)) \to \enh\lambda^\rb_N(K).
\end{equation}
By Lemma~\ref{lem:dtpsiphi} (i), it enters the commutative diagram
\[
\xymatrix@R=2ex@C=2em{
\Eopb\sigma\Eopb o \Enu_N(K) \ar[d]^\wr \ar[r] & \Enu^\rb_N(\Enu_N(K)) \ar[d]^\wr  \ar[r] & \enh\lambda^\rb_N(\Enu_N(K)) \ar[d] \to [+1] \\
\Eopb\sigma\Eopb i_N K \ar[r] & \Enu^\rb_N(K)  \ar[r] & \enh\lambda^\rb_N(K) \to [+1] .
}
\]
Here, the first vertical isomorphism is due to \cite[Lemma~4.8]{DK21a}, and the second vertical isomorphism follows from \cite[Lemma~4.10]{DK21a}.
Hence, also the third vertical arrow is an isomorphism, and the statement follows.

\smallskip\noindent(b)
In order to obtain \eqref{eq:lambdanu}, we are going to connect the relevant spaces in a commutative diagram.

We refer to \cite[\S\S2.3, 2.4]{DK21a} for  notations and  details on the real oriented blow-up $M^\rb_N$, the real projective blow-up $M^\pb_N$, the normal deformation $M^\nd_N$, and the open embedding $M^\nd_N\subset\bl M\times\R\br^\pb_{N\times\st0}$.

With the natural identification $M\times\R_{>0}\simeq\Omega\subset M^\nd_N$, consider the open embeddings (see Figure~\ref{fig:jU})
\[
\xymatrix@R=1ex{
& (M\times\R)^\rb_{N\times\R} \\
\llap{$U\defeq$}\Omega^\rb_{N\times\R_{>0}} \ar@{^(->}[ur]^-j \ar@{^(->}[dr]_-{\widetilde\jmath} \ar@{^(->}[r]|-{\widehat\jmath}
& \bl M^\nd_N\br^\rb_{\overline{N\times\R_{\neq0}}} \\
& \bl(M\times\R)^\rb_{N\times\st0}\br^\rb_{\overline{N\times\R_{\neq0}}}.
}
\]
Here, $\widetilde\jmath$ is induced by the natural embedding $\overline\Omega\subset(M\times\R)^\rb_{N\times\st0}$, compatible with the open embedding $M^\nd_N\subset(M\times\R)^\pb_{N\times\st0}$ of \cite[\S2.4]{DK21a}. More precisely, there is a commutative diagram
\[
\xymatrix{
\overline\Omega \ar@{^(->}[r] \ar@{^(->}[d] & (M\times\R)^\rb_{N\times\st0} \ar[d]_q \\
M^\nd_N \ar@{^(->}[r] & (M\times\R)^\pb_{N\times\st0},
}
\]
where $q$ is the natural projection from the real oriented blow-up to the  projective blow-up.

Let $X$ be the closed subset of $\bl(M\times\R)^\rb_{N\times\st0}\br^\rb_{\overline{N\times\R_{\neq0}}}$ given by
\begin{align*}
X
&\defeq  \overline{\widetilde\jmath(U)} \setminus \widetilde\jmath(U) \\
&= (M\setminus N) \sqcup S_NM \sqcup \bdot T_NM \sqcup S_N(T_NM).
\end{align*}

\begin{figure}
\begin{tikzpicture}[scale=.5,baseline=0]
\filldraw[white!80!black] (-2,0) -- (0,0) -- (0,4) -- (-2,4) -- cycle ;
\filldraw[white!80!black] (1,0) -- (1,4) -- (4,4) -- (4,0) -- cycle ;
\draw[red] (-1,0) -- (-1,4) ;
\draw[red] (2,0) -- (2,4) ;
\draw[red] (3,0) -- (3,4) ;
\draw[red] (0,0) -- (0,4);
\draw[red] (1,0) -- (1,4);
      \draw (4,4) node[below left] {$\overline{j(U)}$};
      \draw[black,very thick] (-2,0) -- (0,0) ;
      \draw[black,very thick] (1,0) -- (4,0) node[below left] {$M^\rb_N$};
\filldraw (0,0) circle (2pt);
\filldraw (1,0) circle (2pt);
\end{tikzpicture}
\qquad
\begin{tikzpicture}[scale=.5,baseline=0]
\begin{scope}
\clip(-2.5,0) rectangle (4.5,4);
\filldraw[white!80!black] (-2.5,0) -- (0,0) -- (0,4) -- (-2.5,4) -- cycle ;
\filldraw[white] (0,0) circle (1);
\filldraw[white!80!black] (1,0) -- (1,4) -- (4.5,4) -- (4.5,0) -- cycle ;
\filldraw[white] (1,0) circle (1);
      \draw[domain=-pi/2+.1:pi/2-.1,smooth,variable=\x,red]  plot ({-1-cos(\x*180/pi)},{sin(\x*180/pi)+tan(\x*180/pi)});
      \draw[domain=-pi/2+.1:pi/2-.1,smooth,variable=\x,red]  plot ({2+cos(\x*180/pi)},{sin(\x*180/pi)+tan(\x*180/pi)});
      \draw[domain=-pi/2+.1:pi/2-.1,smooth,variable=\x,red]  plot ({3+cos(\x*180/pi)},{sin(\x*180/pi)+2*tan(\x*180/pi)});
\end{scope}
\draw[red] (0,1) -- (0,4);
\draw[red] (1,1) -- (1,4);
     \draw (4.5,4) node[below left] {$\overline{\widetilde\jmath(U)}$};
      \draw[black,very thick] (-2.5,0) -- (-1,0) arc (180:90:1) ;
      \draw[black,very thick] (1,1) arc (90:0:1) -- (4.5,0) node[below left] {$X$};
\filldraw (0,1) circle (2pt);
\filldraw (1,1) circle (2pt);
\end{tikzpicture}
\qquad
\begin{tikzpicture}[scale=.5,baseline=0]
\filldraw[white!80!black] (-2,0) -- (0,0) -- (0,4) -- (-2,4) -- cycle ;
\filldraw[white!80!black] (1,0) -- (1,4) -- (4,4) -- (4,0) -- cycle ;
\begin{scope}
\clip(-2,-2) rectangle (4,4);
      \draw[domain=-2:-.1,smooth,variable=\x,red] plot ({\x},{-1/\x});
      \draw[domain=1.1:4,smooth,variable=\x,red] plot ({\x},{1/(\x-1)});
      \draw[domain=1.1:4,smooth,variable=\x,red] plot ({\x},{2/(\x-1)});
\end{scope}
\draw[red] (0,0) -- (0,4);
\draw[red] (1,0) -- (1,4);
      \draw (4,4) node[below left] {$\overline{\widehat\jmath(U)}$};
      \draw[black,very thick] (-2,0) -- (0,0) ;
      \draw[black,very thick] (1,0) -- (4,0) node[below left] {$(T_NM)^\rb_N$};
\filldraw (0,0) circle (2pt);
\filldraw (1,0) circle (2pt);
\end{tikzpicture}
\caption{
The sets $\overline{j(U)}$, $\overline{\widetilde\jmath(U)}$ and $\overline{\widehat\jmath(U)}$ pictured in the case $M=\R$ and $N=\st0$.
The red lines are fibers of the projection $p_U\colon U\to\Omega\to M$.
(Color figure online.)
}
\label{fig:jU}
\end{figure}

Consider the commutative diagram with cartesian squares, where $\widetilde u$ and $u$ are open embeddings,
\[
\xymatrix{
& M \\
& U \ar[u]^-{p_U} \ar[rr]^-{\widehat p} \ar@{^(->}[dl]_-j \ar@{^(->}[d]_-{\widetilde\jmath} \ar@{^(->}[dr]_-{\widehat\jmath} && \Omega \ar[ull]_{p_\Omega} \ar@{^(->}[d]_-{\overline\jmath} \\
\overline{j(U)} & \overline{\widetilde\jmath(U)} \ar[l]^-{\widetilde r} & \overline{\widehat\jmath(U)} \ar@{}[ur]|-\square \ar[r]_-{\widetilde p} \ar@{_(->}[l]^-{\widetilde u} & \overline\Omega \\
M^\rb_N \ar@{}[ur]|-\square \ar@/^13ex/[uuur]^-p \ar@{^(->}[u]_-k & X \ar@{}[ur]|-\square \ar@{^(->}[u]_-{\widetilde k} \ar[l]^-r  & (T_NM)^\rb_N \ar@{}[ur]|-\square \ar[r]_-{\overline p} \ar@{_(->}[l]^-u \ar@{^(->}[u]_-{\widehat k} & T_NM \ar@{^(->}[u]_-{\overline k} \\
S_NM \ar@{}[ur]|-\square \ar@{^(->}[u]_-i & X\setminus(M\setminus N) \ar@{}[ur]|-\square \ar[l]^-{\overline r} \ar@{^(->}[u]_-{\widetilde \ell} & S_N(T_NM). \ar@{_(->}[l]^-{\overline\ell} \ar@{^(->}[u]_-{\overline\imath}
}
\]

Note that $\overline r \circ \overline\ell$ gives the identification $S_N(T_NM)\simeq S_NM$. Hence, by definition, \eqref{eq:lambdanu} is written as
\begin{equation}
\label{eq:lambdanu2}
\Eoim{\overline r} \Eoim{\overline\ell} \Eepb{\overline\imath} \Eopb{\overline p} \Enu_N(K)[1]
\to \Eepb i \Eopb p K[1].
\end{equation}

On one hand, there is a chain of morphisms
\begin{align*}
\Eepb{\overline\imath} \Eopb{\overline p} \Enu_N(K)
&\simeq \Eepb{\overline\imath} \Eopb{\overline p} \Eopb{\overline k} \Eoim{\overline\jmath} \Eopb{p_\Omega} K
\simeq \Eepb{\overline\imath} \Eopb{\widehat k} \Eopb{\widetilde p} \Eoim{\overline\jmath} \Eopb{p_\Omega} K \\
&\underset{(1)}\to \Eepb{\overline\imath} \Eopb{\widehat k} \Eoim{\widehat\jmath} \Eopb{\widehat p} \Eopb{p_\Omega} K
\simeq \Eepb{\overline\imath} \Eopb{\widehat k} \Eoim{\widehat\jmath} \Eopb{p_U} K \\
&\simeq \Eepb{\overline\imath} \Eopb{\widehat k} \Eopb{\widetilde u} \Eoim{\widetilde\jmath} \Eopb{p_U} K
\simeq \Eepb{\overline\imath} \Eopb u \Eopb{\widetilde k} \Eoim{\widetilde\jmath} \Eopb{p_U} K \\
& \underset{(2)}\simeq \Eepb{\overline\ell} \Eepb{\widetilde\ell} \Eopb{\widetilde k} \Eoim{\widetilde\jmath} \Eopb{p_U} K
=  \Eepb{\overline\ell} L ,
\end{align*}
where we set
\[
L \defeq \Eepb{\widetilde \ell} \Eopb{\widetilde k} \Eoim{\widetilde j} \Eopb{p_U} K.
\]
Here, $(1)$ follows by adjunction from the isomorphism $\Eopb{\widehat\jmath}\Eopb{\widetilde p}\Eoim{\overline\jmath} \simeq \Eopb{\widehat p}$, and $(2)$ uses the fact that $\Eopb u \simeq \Eepb u$.

Hence, there is a morphism
\begin{equation}
\label{eq:rll}
\Eoim{\overline r} \Eoim{\overline\ell} \Eepb{\overline\imath} \Eopb{\overline p} \Enu_N(K)
\to \Eoim{\overline r} \Eoim{\overline\ell} \Eepb{\overline\ell} L .
\end{equation}

On the other hand, there is a chain of isomorphisms
\begin{align*}
\Eepb i \Eopb p K
&\underset{(3)}\simeq \Eepb i \Eopb k \Eoim j \Eopb{p_U} K
\simeq \Eepb i \Eopb k \Eoim{\widetilde r} \Eoim{\widetilde j} \Eopb{p_U} K \\
&\underset{(4)}\simeq \Eepb i \Eoim r \Eopb{\widetilde k} \Eoim{\widetilde j} \Eopb{p_U} K
\simeq \Eoim{\overline r} \Eepb{\widetilde \ell} \Eopb{\widetilde k} \Eoim{\widetilde j} \Eopb{p_U} K \\
&\simeq \Eoim{\overline r} L.
\end{align*}
Here, $(3)$ easily follows using the identification $(M\times\R)^\rb_{N\times\R}\simeq(M^\rb_N)\times\R$, and $(4)$ uses the fact that $\widetilde r$ and $r$ are proper.

Hence, the natural morphism $\Eoim{\overline\ell} \Eepb{\overline\ell} L \to L$, combined with \eqref{eq:rll}, induces \eqref{eq:lambdanu2}.
\end{proof}

\subsection{The case of vector bundles}
Let $\tau\colon V\to N$ be a vector bundle, and $o\colon N\to V$ its zero section. Let $\bdot V=V\setminus o(N)$, and consider the quotient
$\gamma\colon \bdot V\to S_NV$ by the $\Gmp$-action.

Consider the projections
\[
\xymatrix{V & V\times_N\bdot V \ar[l]_-{p_1} \ar[r]^-{p_2} & \bdot V}.
\]
For $K\in\BECp V$ and $C\in\BECp{V\times_N\bdot V}$, we set
\begin{align*}
\Phi_C(K) &\defeq \Eeeim{{p_2}}\bl C \ctens \Eopb p_1 K \br, \\
\Psi_C(K) &\defeq \Eoim{{p_2}}\cihom( C , \Eepb p_1 K ).
\end{align*}

\begin{lemma}\label{lem:lambdakernel}
Let $K\in\BECcon V$. With the identifications $N\simeq o(N)\subset V$ and $T_NV\simeq V$, one has
\begin{align*}
\Eopb\gamma\enh\lambda^\rb_N(K) &\simeq \Phi_C(K), \\
\Eopb\gamma\enh\widetilde\lambda^\rb_N(K) &\simeq \Psi_C(K),
\end{align*}
for $C=\epsilon(\field_B)[1]$, with
$B=\st{(x,y)\in V\times_N\bdot V\semicolon x=\lambda y,\text{ for some }\lambda\geq0}$ a closed subset of $V\times_N\bdot V$.
\end{lemma}

\begin{proof}
Since the proofs are similar, let us only discuss the first isomorphism.

Consider the morphisms
\[
\xymatrix@R=3ex{
\bdot V & \bdot V\times\inb{(\R_{\geq 0})} \ar[l]^-{q_1} \ar@/^2.5ex/[rr]^{\tilde p} \ar[r]_-{\tilde\gamma} \ar[dl]^{\tilde\jmath} & \RB V \ar[r]_p & V \\
\bdot V \times \bR \ar[u]_{\tilde q_1} & \bdot V \ar[u]^{i_0} \ar[l]^-{\tilde\imath_0} \ar[r]_\gamma \ar@{}[ur]|\square & S_N V \ar[u]_i,
}
\]
where $i_0(x)=(x,0)$.
One has
\begin{align*}
\Eopb\gamma\enh\lambda^\rb_N(K)
&\simeq \Eepb\gamma\Eepb i\Eopb p K
\simeq \Eepb{i_0}\Eepb{\tilde\gamma}\Eopb p K \\
&\underset{(*)}\simeq \Eepb{i_0}\Eopb{\tilde\gamma}\Eopb p K[1]
\simeq \Eepb{i_0}\Eopb{\tilde p} K[1] \\
&\simeq \Eepb{i_0}\Eepb{\tilde\jmath}\Eeeim{\tilde\jmath}\Eopb{\tilde p} K[1]
\simeq \Eepb{\tilde\imath_0}\Eeeim{\tilde\jmath}\Eopb{\tilde p} K[1] \\
&\underset{(**)}\simeq \Eeeim{{\tilde q_1}}\Eeeim{\tilde\jmath}\Eopb{\tilde p} K[1]
\simeq \Eeeim{{q_1}}\Eopb{\tilde p} K[1] ,
\end{align*}
where $(*)$ is due to the fact that $\tilde\gamma$ is an $\inb{(\Gmp)}$-bundle, and $(**)$ holds because $\Eeeim{\tilde\jmath}\Eopb{\tilde p} K$ is $\bGmp$-conic with respect to the action on the second factor of $\bdot V\times\bR$.

It follows that $\Eopb\gamma\enh\lambda^\rb_N(K) \simeq \Phi_C(K)$
for $C \defeq \reim{(\tilde p,q_1)}\field_{\bdot V\times\R_{\geq0}}[1]$.
Since $(\tilde p,q_1)$ decomposes into
\[
\xymatrix@C=2em{(\tilde p,q_1)\colon\bdot V\times\R_{\geq0}\ar[r]^-\sim & B \ar@{^(->}[r]& V\times_N\bdot V},
\]
we have $\reim{(\tilde p,q_1)}\field_{\bdot V\times\R_{\geq0}}\simeq\field_B$.
\end{proof}

\subsection{Blow-up and vanishing cycles}
Let $X$ be a smooth complex curve, and $a\in X$.
Let $z$ be a coordinate on the complex vector line $T_a X$, and $w$ the dual coordinate on  $T^*_a X$, so that the pairing $T_a X\times T^*_a X\to\C$ is given by $(z,w)\mapsto z w$. Then, the  isomorphism
\[
c\colon \inb{(\bdot T_a X)} \to \inb{(\bdot T^*_a X)}, \quad
z\mapsto -z^{-1}
\]
does not depend on the choice of the coordinate, and induces a homeomorphism
\[
c\colon S_a X\isoto S^*_a X.
\]

\begin{lemma}\label{lem:mutolambda}
For $K\in\BECp X$, there is a natural morphism
\[
\Eopb c\Emu^\sph_{\st 0}(K)[1] \to \enh\lambda^\rb_{\st 0}(K).
\]
\end{lemma}

\begin{proof}
Since $c$ is an isomorphism and $\opb\gamma$ is fully faithful, it is enough to show that there is a natural morphism
\[
\Eopb u\Emu_{\st 0}(K)[1] \to \Eoim c\Eopb\gamma\enh\lambda^\rb_{\st 0}(K),
\]
 where $u\colon\inb{(\bdot T^*_a X)}\to\inb{(T^*_a X)}$ is the natural morphism.
Write $L=\Enu_{\st 0}(K)$. By Lemma~\ref{lem:KnuK2}, it is equivalent to prove that there is a natural morphism in $\BECp{\inb{(\bdot T^*_a X)}}$
\begin{equation}\label{eq:Foulambda1}
\Eopb u L^\wedge[1] \to \Eoim c\Eopb\gamma\enh\lambda^\rb_{\st 0}(L).
\end{equation}
 Set $\V=T_a X$ and $\W=T^*_a X$. 
Consider the subsets of $\V\times\bdot\W$
\[
F=\st{\Re z w \leq 0}, \quad
G=\st{\Re z w \leq 0,\ \Im z w = 0}.
\]
The inclusion of closed subsets $G\subset F$ gives a morphism
\begin{equation}\label{eq:Foulambda2}
\Phi_{\epsilon^+(\field_F)}(L) \to \Phi_{\epsilon^+(\field_G)}(L).
\end{equation}
Then we obtain \eqref{eq:Foulambda1} by applying $\Eopb u$ to
\eqref{eq:Foulambda2}. In fact, on one hand, recalling
the notations on the enhanced Fourier-Sato transforms from
\S\ref{sse:Founumu}, one has
\[
\Eopb u L^\wedge[1] \simeq \Eopb u \Phi_{\epsilon^+(\field_F)}(L)[1].
\]
On the other hand, one has $G\cap(\V\times\bdot \W)=\st{(z,w)\semicolon z=-\lambda w^{-1},\ \exists \lambda\geq 0}$. Hence $\Eoim c\Eopb\gamma\enh\lambda^\rb_{\st 0}(L) \simeq
\Eopb u \Phi_{\epsilon^+(\field_G)}(L)[1]$
by Lemma~\ref{lem:lambdakernel}.
\end{proof}

\begin{proposition}
\label{pro:phinu'S}
Let $K\in\Eperv(\ifield_X)$. Then there are natural isomorphisms in $\BECp{S_a X}$
\[
\enh\widetilde\lambda_{\st a}^\rb(K) \isofrom
\Eopb c \Emu^\sph_{\st a}(K)[1] \isoto
\enh\lambda_{\st a}^\rb(K).
\]
In particular, $\enh\lambda_{\st a}^\rb(K)\simeq\enh\widetilde\lambda_{\st a}^\rb(K)$ is of sheaf type, and its associated sheaf is a local system.
\end{proposition}

\begin{proof}
(i) Let us show that the first isomorphism follows by duality from the second one.

One has
\begin{align*}
\enh\widetilde\lambda_{\st a}^\rb(K)
&\simeq \enh\widetilde\lambda_{\st a}^\rb(\Edual \Edual K)
\underset{(*)}\simeq \Edual \bl\enh\lambda_{\st a}^\rb(\Edual K) [-1] \br, \\
\Eopb c \Emu^\sph_{\st a}(K)
&\simeq \Eopb c \Emu^\sph_{\st a}(\Edual \Edual K)
\underset{(**)}\simeq \Edual \bl\Eopb c \Emu^\sph_{\st a}(\Edual K)[-1]\br.
\end{align*}
where $(*)$ follows from Lemma~\ref{lem:rbdual}, and
$(**)$ from \cite[Lemma~4.5]{DK21a}.

\smallskip\noindent(ii) Let us prove the first isomorphism.
Set $\V=T_a X$ and $\W=T^*_a X$. By Proposition~\ref{pro:numuperv}, one has $\Enu_{\st a}(K)\simeq e\iota(F)$ for some $F\in\Perv(\field_\V)\cap\BDC_{\Gmp}(\field_\V)$.
By \cite[Lemma~4.10]{DK21a} and Lemma~\ref{lem:KnuK2}, we may take $X=\V$, $a=0$, and $K=e\iota(F)$.
Hence, we are reduced to prove that the morphism
\[
\Eopb c \Emu^\sph_{\st 0}(e\iota(F))[1] \to
\enh\lambda_{\st 0}^\rb(e\iota(F)),
\]
from Lemma~\ref{lem:mutolambda}, is an isomorphism. One has
\begin{align*}
\Eopb c \Emu^\sph_{\st 0}(e\iota(F))
&\simeq e\bl \opb c \mu^\sph_{\st 0}(F)\br , \\
\enh\lambda_{\st 0}^\rb(e\iota(F))
&\simeq e\bl \lambda_{\st 0}^\rb(F)\br.
\end{align*}
Since $e\iota$ is fully faithful, it is enough to show that there is an isomorphism
\[
\opb c \mu^\sph_{\st 0}(F)[1] \isoto
\lambda_{\st 0}^\rb(F),
\]
which can be checked at the level of stalks.

The underlying real vector spaces to $\V$ and $\W$ are in duality by the pairing $\langle v,w \rangle = \Re(z w)$.
For $\Gamma\subset\W$, the set
\[
\Gamma^\circ = \{v\in\V \semicolon \langle v,w \rangle \geq0 \text{ for any } w\in\Gamma \}
\]
is called the polar cone of $\Gamma$.

For $\theta\in S_a X=S_0\V$, by \cite[Theorems~4.2.3, 4.3.2]{KS90} one has
\begin{align*}
\bl \nu_{\st 0}^\sph(F) \br_\theta
&\simeq  \ilim[\Lambda,r]\RHom(\field_{\Lambda\cap \{|z|<r\}} ,F) \\
&\underset{(*)}\simeq  \ilim[\Lambda]\RHom(\field_\Lambda ,F), \\
\bl\mu^\sph_{\st 0}(F)\br_{c(\theta)}
&\simeq  \ilim[\Gamma,r]\RHom(\field_{\Gamma^\circ\cap \{|z|<r\}} ,F) \\
&\underset{(*)}\simeq  \ilim[\Gamma]\RHom(\field_{\Gamma^\circ} ,F),
\end{align*}
where $\Lambda$ runs over the open convex proper cones in $\V$
containing $\theta$, $\Gamma$ runs over the open convex proper cones in $\W$
containing $c(\theta)$, and $r\to0+$. Here, the isomorphisms $(*)$ are due to the fact that $F$ is conic.

It then follows from Lemma~\ref{lem:dtpsiphi}~(i) that one has
\[
\bl \lambda_{\st 0}^\rb(F) \br_\theta[-1]
\simeq  \ilim[\Lambda]\RHom(\field_{\V\setminus\Lambda} ,F).
\]

For any $\Lambda$ as above, taking $\Gamma = c(\Lambda)$, one has $\lambda_2\defeq\V\setminus\Lambda \supset \Gamma^\circ\eqdef\lambda_1$.
Hence, it is enough to prove
\[
\RHom(\field_{\lambda_2 \setminus \lambda_1} ,F)\simeq 0.
\]
Consider the maps
\[
\xymatrix{
\V & \bdot\V \ar@{_(->}[l]_-j \ar[r]^-q & S_0\V=\bdot\V/\Gmp.
}
\]
Let $\emptyset\subsetneq I_k\subsetneq S_0\V$ be the closed connected subset such that $\bdot\lambda_k=\opb q(I_k)$, for $k=1,2$.
Let $L$ be a local system on $S_0V$ such that $\opb j F\simeq \opb q L[1]$.
Then, one has
\begin{align*}
\RHom(\field_{\lambda_2 \setminus \lambda_1} ,F)
&\simeq \RHom(\reim j\field_{\bdot\lambda_2 \setminus \bdot\lambda_1} ,F)
\simeq \RHom(\field_{\bdot\lambda_2 \setminus \bdot\lambda_1} ,\opb q L[1]) \\
&\simeq \RHom(\field_{\bdot\lambda_2 \setminus \bdot\lambda_1} ,\epb q L)
\simeq \RHom(\reim q\field_{\bdot\lambda_2 \setminus \bdot\lambda_1} ,L) \\
&\simeq \RHom(\field_{I_2 \setminus I_1}[-1] ,L).
\end{align*}
The last term vanishes, since
\[
V\simeq\RHom(\field_{I_1}[-1] ,L) \isoto \RHom(\field_{I_2}[-1] ,L)\simeq V,
\]
where $V$ is the stalk of $L$.
\end{proof}

\begin{remark}
Let us use notations as in Lemma~\ref{lem:mutolambda} and its proof.
Then, for $L=\Enu_{\st 0}K$ with $K\in\Eperv(\ifield_X)$, the morphism \eqref{eq:Foulambda2} is an isomorphism. In fact, on one hand, $\Eopb u\eqref{eq:Foulambda2}= \eqref{eq:Foulambda1}$ is an isomorphism by Proposition~\ref{pro:phinu'S}.
On the other hand,
\[
\Eopb o \Phi_{\epsilon^+(\field_F)}(L) \simeq
\Eoim \tau L \simeq
\Eopb o \Phi_{\epsilon^+(\field_G)}(L),
\]
where $o\colon \st 0\to\V$ is the linear embedding, and $\tau\colon\V\to\st 0$ its transpose.
\end{remark}

\subsection*{Acknowledgements}
The research of A.D'A.\
was partially supported by GNAMPA/INdAM.  He acknowledges the kind hospitality at RIMS of
Kyoto University, and at the Perimeter Institute in Waterloo, during the preparation of this paper.

The research of M.K.\
was supported by Grant-in-Aid for Scientific Research (B)
15H03608, Japan Society for the Promotion of Science

\end{document}